\documentclass[10pt]{article}

\font\smallit=cmti10

\usepackage{amssymb,latexsym,amsmath,epsfig,amsthm}
\usepackage{rgdw_global, rgdw}
\usepackage[a-1b]{pdfx}
\usepackage{hyperref}
\usepackage{graphicx}
\author{Yuki Irie}
\date{\today}
\title{}
\begin{document}

\begin{center}
\uppercase{\bf A base-\(p\) Sprague-Grundy type theorem for \(p\)-calm subtraction games: Welter's game and representations of generalized symmetric groups}
\vskip 20pt
{\bf Yuki Irie\footnote{This work was supported by JSPS KAKENHI Grant Number JP20K14277.}}\\
{\smallit Research Alliance Center for Mathematical Sciences, Tohoku University, Miyagi, Japan}\\
{\tt yirie@tohoku.ac.jp}\\
\end{center}
\vskip 20pt

\centerline{\bf Abstract}

\noindent
For impartial games \(\Gamma\) and \(\Gamma'\), the Sprague-Grundy function of the disjunctive sum \(\Gamma + \Gamma'\) 
is equal to the Nim-sum of their Sprague-Grundy functions.
In this paper, we introduce \(p\)-calm subtraction games, 
and show that for \(p\)-calm subtraction games \(\Gamma\) and \(\Gamma'\),
the Sprague-Grundy function of a \(p\)-saturation of \(\Gamma + \Gamma'\) 
is equal to the \(p\)-Nim-sum of the Sprague-Grundy functions of their \(p\)-saturations.
Here a \(p\)-Nim-sum is the result of addition without carrying in base \(p\) and a \(p\)-saturation of \(\Gamma\) is an impartial game obtained from \(\Gamma\) by adding some moves.
It will turn out that Nim and Welter's game are \(p\)-calm.
Further, using the \(p\)-calmness of Welter's game, we generalize a relation between Welter's game and representations of symmetric groups
to disjunctive sums of Welter's games and representations of generalized symmetric groups;
this result is described combinatorially in terms of Young diagrams.

\pagestyle{myheadings}
\markright{}
\thispagestyle{empty}
\baselineskip=12.875pt
\vskip 30pt 

\section{Introduction}
\label{sec:org3a2010e}
Base 2 plays a key role in combinatorial game theory. 
Specifically, the Sprague-Grundy function of the disjunctive sum of two impartial games is equal to the Nim-sum of their Sprague-Grundy functions.
Here a Nim-sum is the result of addition without carrying in base 2.
In particular, the Sprague-Grundy value of a position in Nim equals the Nim-sum of the heap sizes.
It is rare that the Sprague-Grundy function of an impartial game can be written explicitly like Nim.
Another well-known example is Welter's game, a subtraction game like Nim;
Welter \cite{welter-theory-1954} gave an explicit formula for its Sprague-Grundy function by using the binary numeral system (see Theorem \ref{orga5b1b7f}).

A few games related to base \(p\) have also been found,
where \(p\) is an integer greater than 1, not necessarily prime.
For example, Flanigan found a game, called \(\text{Rim}_p\);
the Sprague-Grundy value of a position in \(\text{Rim}_p\) equals the \(p\)-Nim-sum of the heap sizes, where a \(p\)-Nim-sum is the result of addition without carrying in base \(p\).\footnote{\(\text{Rim}_{\ccp }\) was devised by James A. Flanigan in an unpublished paper entitled ``NIM, TRIM and RIM.''}
We use \(\oplus_p\) for the \(p\)-Nim-sum.\footnote{The operation \(\oplus_p\) is different from that related to Moore's \(\text{Nim}_{k - 1}\) \cite{moore-generalization-1910} and Li's \(k\)-person Nim \cite{li-nperson-1978}. These games were analyzed using addition modulo \(\cck\) in base 2.}
For example, consider the heap \((3, 7, 4)\).
While, in Nim, the Sprague-Grundy value of \((3, 7, 4)\) is equal to
\[
3 \oplus_2 7 \oplus_2 4 = (1 + 2) \oplus_2 (1 + 2 + 4) \oplus_2 (4) = 0,
\]
in \(\text{Rim}_3\), its Sprague-Grundy value is equal to
\[
 3 \oplus_3 7 \oplus_3 4 = (3) \oplus_3 (1 + 2 \cdot 3) \oplus_3 (1 + 3) = 2 + 3 = 5.
\]
Thus we can say that \(\text{Rim}_p\) is a base-\(p\) version of Nim.
Irie \cite{irie-pSaturations-2018a} observed that there are infinitely many base-\(p\) versions of Nim.
From this observation, he introduced \emph{\(p\)-saturations} and showed that a \(p\)-saturation of Nim is a base-\(p\) version of Nim.
That is, the Sprague-Grundy value of a position in a \(p\)-saturation of Nim equals the \(p\)-Nim-sum of the heap sizes.
Figure \ref{fig:org5cb71dd} shows an example of a 3-saturation of Nim.
While we can take tokens from just one heap in Nim, 
it is allowed to take tokens from multiple heaps with a restriction in a \(p\)-saturation of Nim.
Incidentally, \(\text{Rim}_p\) is one of the \(p\)-saturations of Nim,
and \(p\)-saturations are defined for subtraction games (see Section \ref{org6e15e26} for details).

\begin{figure}[htbp]
\centering
\includegraphics[scale=0.83]{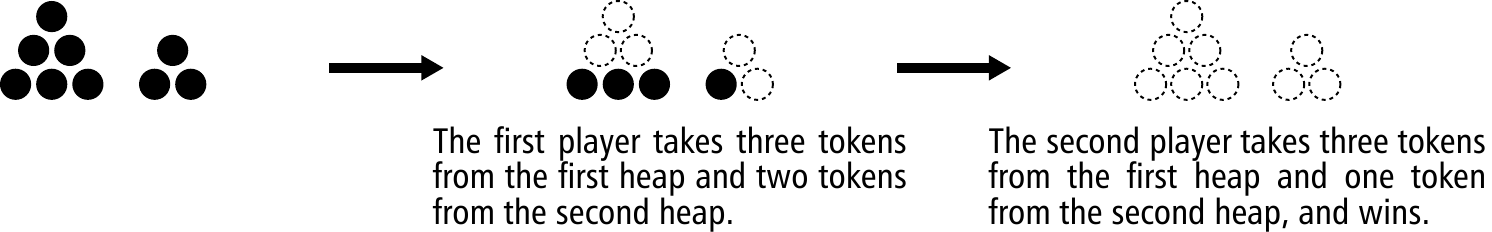}
\caption{\label{fig:org5cb71dd}An example of a 3-saturation of Nim.}
\end{figure}

\noindent
Further, it was shown that a \(p\)-saturation of Welter's game is a base-\(p\) version of Welter's game \cite{irie-pSaturations-2018a}.
In other words, we can obtain an explicit formula for the Sprague-Grundy function of a \(p\)-saturation of
Welter's game by rewriting Welter's formula with base \(p\) (see Theorem \ref{org28c6b37}).

In this paper, an \emph{impartial game} is defined to be a digraph such that the maximum length of a walk from each vertex is finite.
We will recall the basics of impartial games in Section \ref{org70b71d2}.
Let \(\Gamma^1\) and \(\Gamma^2\) be two impartial games.
Then the Sprague-Grundy function of the disjunctive sum \(\Gamma^1 + \Gamma^2\) is equal to the Nim-sum of their Sprague-Grundy functions.
The fundamental question of this paper is whether there exists an operation \(+_p\) such that the Sprague-Grundy function of \(\Gamma^1 +_p \Gamma^2\) is equal to the \(p\)-Nim-sum of their Sprague-Grundy functions.
For example, if \(p = 2\), then the ordinary disjunctive sum satisfies this condition.
We present a partial solution to the question.
More precisely, we consider the following condition on subtraction games \(\Gamma^1\) and \(\Gamma^2\):
\begin{description}
\item[{(PN)}] The Sprague-Grundy function of a \(p\)-saturation of \(\Gamma^1 + \Gamma^2\) is equal to the \(p\)-Nim-sum of the Sprague-Grundy functions of their \(p\)-saturations.
\end{description}
For a subtraction game \(\Gamma^1\) satisfying a saturation condition, we first give a necessary condition for \(\Gamma^1\) to satisfy (PN) when \(\Gamma^2\) is Nim (Lemma \ref{orgf5a2cb9}).
If \(\Gamma^1\) satisfies this necessary condition, then \(\Gamma^1\) will be said to be \emph{\(p\)-calm}.
It will turn out that Nim and Welter's game are \(p\)-calm.
Our main theorem (Theorem \ref{org25184ea}) states that \(p\)-calm subtraction games satisfy (PN) and are closed under disjunctive sum.
In particular, \(\Gamma^1\) and Nim satisfy (PN) if and only if \(\Gamma^1\) is \(p\)-calm.

Using \(p\)-calmness of Welter's game, we can generalize a relation between 
Welter's game and representations of symmetric groups;
this result is described combinatorially in Section \ref{org3ea736a}, and its algebraic interpretation is stated in Remark \ref{org09152e2}.
Sato \cite{sato-game-1968, sato-mathematical-1970, sato-maya-1970} studied Welter's game independently from Welter.
Although Welter's game is usually described as a coin-on-strip game,
Sato realized that this game can be considered as a game with Young diagrams, as we will describe in Section \ref{org1dc745b}.
He then found that the Sprague-Grundy function of Welter's game
can be written in a similar way to the hook formula, which is a formula for representations of symmetric groups (Theorems \ref{orgd359077} and \ref{org94584db}).
From this, Sato conjectured that Welter's game is related to representations of symmetric groups.
A relation between them was discovered by Irie \cite{irie-pSaturations-2018a}.
Specifically, for a prime \(\ccp\), he showed a theorem on representations of symmetric groups, 
which will be called the \(p'\)-component theorem (Theorem \ref{org2374258});
by using this theorem, he obtained an explicit formula for the Sprague-Grundy function of a \(p\)-saturation of Welter's game.\footnote{Although the \(p'\)-component theorem holds only when \(\ccp\) is prime, a slightly weaker result holds even when \(\ccp\) is not prime. By using this, we can obtain an explicit formula for the Sprague-Grundy function of a \(p\)-saturation of Welter's game.}
Here, the \(p'\)-component theorem is a result of representations of symmetric groups with degree prime to \(p\).
Incidentally, on representations with degree prime to \(p\), there is a famous conjecture of McKay, 
which is one of the most important conjectures in representation theory (see, for example, \cite{navarro-character-2018} for details).
In the present article, we generalize the \(p'\)-component theorem to disjunctive sums of Welter's games and representations of generalized symmetric groups, 
which will be described combinatorially in terms of Young diagrams (Theorem \ref{org4fefb85}).

This paper is organized as follows.
Section \ref{org70b71d2} contains the basics of impartial games.
In Section \ref{org6e15e26}, we recall \(p\)-saturations and define \(p\)-calm subtraction games.
We then prove that \(p\)-calm subtraction games satisfy the condition (PN) (Theorem \ref{org25184ea});
moreover, using the \(p\)-calmness of Welter's game, we generalize a property of Welter's game (Proposition \ref{org14ff579}), which yields a generalization of the \(p'\)-component theorem in Section \ref{org3ea736a}.

\section{Subtraction Games}
\label{sec:org46e57ad}
\label{org70b71d2}
This section provides the basics of impartial games.
We define subtraction games, disjunctive sums, and Sprague-Grundy functions.
See \cite{berlekamp-Winning-2001, conway-numbers-2001} for more details of combinatorial game theory.

We first introduce some notation for impartial games.
Let \(\Gamma\) be a digraph with vertex set \(\Position[\Gamma]\) and edge set \(\Edge[\Gamma]\),
that is, \(\Position[\Gamma]\) is a set and \(\Edge[\Gamma] \subseteq \Position[\Gamma]^2\).
As we have defined in the introduction,
the digraph \(\Gamma\) is called a (\emph{short}) \emph{impartial game} if
the maximum length \(\lg_{\Gamma}(\ccA)\) of a walk from each vertex \(\ccA\) is finite.
Let \(\Gamma\) be an impartial game. The vertex set \(\Position[\Gamma]\) is called its \emph{position set}.
Let \(\ccA\) and \(\ccB\) be two positions in \(\Gamma\).
If \((\ccA, \ccB) \in \Edge[\Gamma]\), then \(\ccB\) is called an \emph{option} of \(\ccA\).
If there exists a path from \(\ccA\) to \(\ccB\), then \(\ccB\) is called a \emph{descendant} of \(\ccA\).
A descendant \(\ccB\) of \(\ccA\) is said to be \emph{proper} if \(\ccB \neq \ccA\).

 \begin{example}
 \comment{Exm.}
\label{sec:orga344f8a}
\label{org84aa7b1}
The digraph with vertex set \(\set{1,2,3}\) and edge set \(\set{(1,2), (2,3), (1,3)}\)
is an impartial game. However, the digraph with vertex set \(\set{1,2}\) and edge set \(\set{(1,2), (2,1)}\)
is not an impartial game since it has the walk \((1,2,1,2,\ldots)\) of infinite length.
 
\end{example}

\begin{remark}
 \comment{Rem.}
\label{sec:org9c4bc25}
\label{org4d08028}
Let \(\Gamma\) be an impartial game with at least one position.
We can consider \(\Gamma\) as the following two-player game.
Before the game, we put a token on a starting position \(\ccA \in \Position[\Gamma]\).
The first player moves the token from \(\ccA\) to its option \(\ccB\).
Similarly, the second player moves the token from \(\ccB\) to its option \(\ccC\).
In this way, the two players alternately move the token.
The winner is the player who moves the token last.
For example, let \(\Gamma\) be the impartial game with position set \(\set{1,2,3,4}\)
and edge set \(\set{(1,2), (2,3), (1,4)}\), and start at position 1.
The first player can move the token to either position 2 or 4.
If she moves it to position 2, then the second player moves it to 3 and wins.
Thus she should move the token to 4.
 
\end{remark}

\comment{connect}
\label{sec:org7c97ee6}

We now define subtraction games.
Let \(\NN\) be the set of nonnegative integers.
Elements in \(\NN^\ccm\) will be denoted by upper-case letters, and components of them by lower-case
letters with superscripts. For example, \(\ccA = (\cca^1, \ldots, \cca^\ccm) \in \NN^\ccm\).
Let \(\Position \subseteq \NN^m\) and \(\sC \subseteq \NN^m \setminus \set{(0,\ldots, 0)}\).
Define \(\Gamma(\Position, \sC)\) to be the impartial game with position set \(\Position\)
and edge set
\[
 \set{(\ccA, \ccB) \in \Position^2 : \ccA - \ccB \in \sC}.
\]
The game \(\Gamma(\Position, \sC)\) is called a \emph{subtraction game}.

 \begin{example}
 \comment{Exm.}
\label{sec:orge836e2a}
\label{org07587d8}
Let
\[
 \WeightOneSset{\ccm} = \Set{\ccC \in \NN^m : \wt(\ccC) = 1},
\]
where \(\wt(\ccC)\) is the Hamming weight of \(\ccC\), that is, the number of nonzero components of \(\ccC\).
The subtraction game \(\Gamma(\NN^{\ccm}, \WeightOneSset{\ccm})\) is called \emph{Nim} and is denoted by \(\Nim[\ccm]\).
For example, in \(\Nim[2]\), the options of \((1,2)\) are \((0,2)\), \((1,1)\), and \((1,0)\).
 
\end{example}

 \begin{example}
 \comment{Exm.}
\label{sec:org8d4bc98}
\label{orgbfa3168}
Let 
\[
 \Position = \set{\ccA \in \NN^m : \cca^i \neq \cca^j \tfor 1 \le \cci < \ccj \le \ccm}.
\]
The subtraction game \(\Gamma(\Position, \WeightOneSset{\ccm})\) is called \emph{Welter's game} and is denoted by \(\Welter[\ccm]\).
For example, in \(\Welter[2]\), the options of \((1,2)\) are \((0,2)\) and \((1,0)\).
Note that
\begin{equation}
\label{org84a24bd}
\lg_{\Welter[\ccm]}(\ccA) = \sum_{\cci = 1}^\ccm \cca^\cci - (1 + 2 + \cdots + \ccm - 1) = \sum_{\cci = 1}^\ccm \cca^\cci - {m \choose 2}.
\end{equation}
For example, if \(\ccA = (1, 3, 4)\), then
\[
 ((1, 3, 4),\ \ (0, 3, 4),\ \ (0, 2, 4),\ \ (0, 1, 4),\ \ (0, 1, 3),\ \ (0, 1, 2))
\]
has length \(5\) (\(=1 + (3 - 1) + (4 - 2)\)).
 
\end{example}

\comment{connect}
\label{sec:org6578092}
We define disjunctive sums.
For \(\cci \in \set{1, 2}\), let \(\Gamma^i\) be an impartial game,
and let \(\Position^i = \Position[\Gamma^i]\) and \(\Edge^i = \Edge[\Gamma^i]\).
The \emph{disjunctive sum} \(\Gamma^1 + \Gamma^2\) of \(\Gamma^1\) and \(\Gamma^2\) is defined
to be the impartial game with position set \(\cP^1 \times \cP^2\) and
edge set
\begin{align*}
 &\set{((\ccA^1, \ccA^2), (\ccB^1, \ccA^2)) : (\ccA^1, \ccB^1) \in \cE^1,\ \ccA^2 \in \cP^2}\\
 \cup &\set{((\ccA^1, \ccA^2), (\ccA^1, \ccB^2)) : (\ccA^2, \ccB^2) \in \cE^2,\ \ccA^1 \in \cP^1}.
\end{align*}
For example, 
\[
 \Nim[\ccm] = \underbrace{\Nim[1] + \cdots + \Nim[1]}_\ccm.
\]
Note that the disjunctive sum of two subtraction games is again a subtraction game.
Indeed, if \(\cP^i \subseteq \NN^{m^i}\) and \(\Gamma^i = \Gamma(\cP^i, \sC^i)\) for \(i \in \set{1,2}\),
then \(\Gamma^1 + \Gamma^2 = \Gamma(\cP^1 \times \cP^2, \sC)\), where
\begin{align*}
 \sC = &\set{(c^{1,1}, \ldots, c^{1, m^1}, \overbrace{0, \ldots, 0}^{m^2}) : (c^{1,1}, \ldots, c^{1, m^1}) \in \sC^1}\\
  \cup &\set{(\underbrace{0, \ldots, 0}_{m^1}, c^{2,1}, \ldots, c^{2, m^2}) : (c^{2,1}, \ldots, c^{2, m^2}) \in \sC^2}.
\end{align*}

\comment{connect}
\label{sec:org4220974}
We define Sprague-Grundy functions.
Let \(\Gamma\) be an impartial game.
For \(\ccA \in \Position[\Gamma]\), the \emph{Sprague-Grundy value} \(\sg_{\Gamma}(\ccA)\) of \(\ccA\) is defined by
\[
 \sg_{\Gamma}(\ccA) = \mex \Set{\sg_{\Gamma}(\ccB) : (\ccA, \ccB) \in \Edge[\Gamma]},
\]
where \(\mex \ccS = \min \set{\alpha \in \NN : \alpha \not \in \ccS}\).
The function \(\sg_{\Gamma}: \Position[\Gamma] \to \NN\) is called the \emph{Sprague-Grundy function} of \(\Gamma\).
An easy induction shows that the second player can force a win if and only if the Sprague-Grundy value of the starting position equals 0.

 \begin{theorem}[\hspace{-0.001em}\cite{sprague-uber-1935, grundy-mathematics-1939}]
 \comment{Thm. [\hspace{-0.001em}\cite{sprague-uber-1935, grundy-mathematics-1939}]}
\label{sec:orgb8714c4}
\label{org4cd79b6}
If \(\Gamma^1\) and \(\Gamma^2\) are impartial games, then
\[
 \sg_{\Gamma^1 + \Gamma^2} = \sg_{\Gamma^1} \oplus_2 \sg_{\Gamma^2},
\]
that is, if \(\ccA^i\) is a position in \(\Gamma^i\) for \(i \in \set{1, 2}\), then
\[
 \sg_{\Gamma^1 + \Gamma^2} (\ccA^1, \ccA^2) = \sg_{\Gamma^1}(\ccA^1) \oplus_2 \sg_{\Gamma_2}(\ccA^2).
\]
 
\end{theorem}

 \begin{example}
 \comment{Exm.}
\label{sec:org48e0d35}
\label{org578838c}
Since \(\Nim[\ccm] = \Nim[1] + \cdots + \Nim[1]\), it follows from Theorem \ref{org4cd79b6} that
\[
 \sg_{\Nim[\ccm]}(\ccA) = \cca^1 \oplus_2 \cdots \oplus_2 \cca^\ccm \quad \tfor \ccA \in \NN^\ccm.
\]
 
\end{example}

 \begin{theorem}[\hspace{-0.001em}\cite{welter-theory-1954}]
 \comment{Thm. [\hspace{-0.001em}\cite{welter-theory-1954}]}
\label{sec:orgc022034}
\label{orga5b1b7f}
If \(\ccA\) is a position in Welter's game \(\Welter[\ccm]\), then
\[
 \sg_{\Welter[\ccm]}(\ccA) = \cca^1 \oplus_2 \cdots \oplus_2 \cca^\ccm \oplus_2 \left(\bigopluspm[i < j][][2][-0.3em][0.3em] 2^{\ord_2(\cca^i - \cca^j) + 1} - 1 \right),
\]
where \(\ord_2(\cca)\) is the \(2\)-adic order of \(\cca\), that is,
\[
 \ord_2(\cca) = \begin{cases}
 \max \set{\ccL \in \NN : 2^\ccL \mid \cca} & \tif \cca \neq 0, \\
 \infty & \tif \cca = 0.
 \end{cases}
\]
 
\end{theorem}

 \begin{example}
 \comment{Exm.}
\label{sec:orgfb67911}
\label{org27df6bf}
Let \(\ccA\) be the position \((7, 5, 3)\) in \(\Welter[3]\). 
By Theorem \ref{orga5b1b7f}, we see that
\begin{align*}
 \sg_{\Welter[3]}(\ccA) &= 7 \oplus_2 5 \oplus_2 \oplus 3 \oplus_2 (2^{\ord_2(7 - 5) + 1} - 1) \oplus_2 (2^{\ord_2(7 - 3) + 1} - 1) \oplus_2 (2^{\ord_2(5 - 3) + 1} - 1) \\
 &= 1 \oplus_2 3 \oplus_2 7 \oplus_2 3 = 6.
\end{align*}
 
\end{example}

\section{\(p\)-Calm Subtraction Games}
\label{sec:org90bb760}
\label{org6e15e26}
We first recall \(p\)-saturations.
We then give a necessary condition for \(\Gamma^1\) to satisfy (PN) when \(\Gamma^2 = \Nim[1]\) (Lemma \ref{orgf5a2cb9}),
and define \(p\)-calm subtraction games.
We next prove that \(p\)-calm subtraction games satisfy (PN) and are closed under disjunctive sum (Theorem \ref{org25184ea}).
Finally, using the \(p\)-calmness of Welter's game, we generalize a property of Welter's game (Proposition \ref{org14ff579}).

\subsection{Notation}
\label{sec:org7c9f61e}
Fix an integer \(\ccp\) greater than 1, and let \(\Omega = \set{0,1,\ldots, p - 1}\).
For \(\cca, L \in \NN\), let \(\cca_L\) denote the \(L\)th digit in the \(p\)-adic expansion of \(\cca\).
Then
\[
 \cca = \sum_{L \in \NN} \cca_L p^L, \quad \cca_L \in \Omega.
\]
We write \(\cca = \pexp<\ccp>{\cca_0, \cca_1, \ldots}\).
If \(\ccN \in \NN\) and \(\cca_{L} = 0\) for \(\ccL \ge \ccN\), then we also write \(\cca = \pexp<\ccp>{\cca_0, \cca_1, \ldots, \cca_{\ccN - 1}}\) .
For example, if  \(p = 2\) and \(\cca = 14\), then \(\cca = \pexp<2>{0, 1, 1, 1, 0, 0, \ldots} = \pexp<2>{0, 1, 1, 1}\).
For \(\ccx, \ccy \in \Omega\), we write \(\ccx \oplus \ccy = \ccx \oplus_\ccp \ccy\) and
\(\ccx \ominus \ccy = \ccx \ominus_\ccp \ccy\), where \(\ominus_\ccp\) is the subtraction without borrowing in base \(\ccp\).
For example, if \(p = 5\), then \(2 \ominus_5 4 = 3\) and
\[
 13 \ominus_5 16 = \pexp<5>{3, 2} \ominus_5 \pexp<5>{1, 3} = \pexp<5>{3 \ominus 1, 2 \ominus 3} = \pexp<5>{2, 4} = 22.
\]

Before proceeding, we present a simple lemma.
 \begin{lemma}
 \comment{Lem.}
\label{sec:org7589b5c}
\label{org912a925}
For \(\cci \in \set{1, \ldots, \ccm}\), let \(\cca^\cci\) and \(\ccb^\cci\) be nonnegative integers.
If \(\cca^\cci \equiv \ccb^\cci \pmod{p^\ccN}\), then
\begin{equation}
\label{org212e187}
  \sum_\cci \cca^\cci - \ccb^\cci \equiv \bigopluspm[\cci][][\ccp][-0.35em][0.35em] \cca^\cci \ominus_\ccp \ccb^\cci \equiv \pexp<\ccp>{0, \ldots, 0, \bigoplus_\cci \cca^\cci_\ccN \ominus \ccb^\cci_\ccN} \pmod{p^{N + 1}}.
\end{equation}
 
\end{lemma}

\begin{proof}
 \comment{Proof.}
\label{sec:org2884fd4}
Since \(\cca^\cci \equiv \ccb^\cci \pmod{p^\ccN}\), it follows that
\[
 \cca^\cci - \ccb^\cci \equiv \cca^\cci \ominus_\ccp \ccb^\cci \equiv \pexp<\ccp>{0, \ldots, 0, \cca^\cci_\ccN \ominus \ccb^\cci_\ccN} \pmod{p^{\ccN + 1}}.
\]
Hence (\ref{org212e187}) holds.
\end{proof}

\subsection{\(p\)-Saturations}
\label{sec:org3423156}
We define \(p\)-saturations.
For \(\cca \in \NN\), let \(\ord_p(\cca)\) denote the \(\ccp\)-adic order of \(\ccn\), that is,
\[
 \ord_\ccp(\cca) = \begin{cases}
 \max \set{\ccL \in \NN : \ccp^\ccL \mid \cca} & \tif \cca \neq 0, \\
 \infty & \tif \cca = 0.
 \end{cases}
\]
For example, \(\ord_2(12) = \ord_2(\pexp<2>{0,0,1,1}) = 2\) and \(\ord_3(12) = \ord_3(\pexp<3>{0,1,1}) = 1\).
Define
\[
 \SatSset{\ccm}{\ccp} = \Set{\ccC \in \NN^\ccm \setminus \set{(0,\ldots,0)}: \ord_p \left(\sum_{\cci} \ccc^i \right) = \mord_p(\ccC)},
\]
where \(\mord_p(\ccC) = \min \set{\ord_p(\ccc^i) : 1 \le \cci \le \ccm}\).
For example, 
\[
 (1, 0), (1, 3) \in \SatSset{2}{3} \quad \tand \quad (1, 2) \not \in \SatSset{2}{3} 
\]
because
\begin{align*}
 \ord_3(1 + 0) = 0 &= \min \set{\ord_3(1), \ord_3(0)} = \min \set{0, \infty}, \\
 \ord_3(1 + 3) = 0 &= \min \set{\ord_3(1), \ord_3(3)} = \min \set{0, 1}, \tand\\
 \ord_3(1 + 2) = 1 &> 0 = \min \set{\ord_3(1), \ord_3(2)} = \min \set{0, 0}.
\end{align*}
Note that \(\WeightOneSset{\ccm} \subseteq \SatSset{\ccm}{\ccp}\).
Let \(\Gamma = \Gamma(\cP, \sC)\) and \(\tilde{\Gamma} = \Gamma(\cP, \tilde{\sC})\).
The game \(\tilde{\Gamma}\) is called a \mbox{\emph{\(p\)-saturation}}\footnote{The definition of \(p\)-saturations is slightly generalized from that in \cite{irie-pSaturations-2018a}.} of \(\Gamma\) if
it has the same Sprague-Grundy function as \(\Gamma(\cP, \sC \cup \SatSset{\ccm}{\ccp})\), that is,
\begin{equation}
\label{org9902293}
 \sg_{\tilde{\Gamma}}(\ccA) = \sg_{\Gamma(\cP, \sC \cup \SatSset{\ccm}{\ccp})}(\ccA)
\end{equation}
for every \(\ccA \in \ccP\).
It is clear that \(\Gamma(\cP, \sC \cup \SatSset{\ccm}{\ccp})\) is a \(p\)-saturation of \(\Gamma\).
In this paper, we will consider a subtraction game \(\Gamma(\cP, \sC)\) satisfying
\begin{description}
\item[{(*)}] \(\sC \subseteq \SatSset{\ccm}{\ccp}\).
\end{description}
Note that if \(\Gamma\) satisfies (*), then \(\tilde{\Gamma}\) is a \(p\)-saturation of \(\Gamma\)
if and only if \(\sg_{\tilde{\Gamma}} = \sg_{\Gamma(\cP, \SatSset{\ccm}{\ccp})}\).
Moreover, if two subtraction games \(\Gamma^1\) and \(\Gamma^2\) satisfy (*), then so does \(\Gamma^1 + \Gamma^2\).

It is known that we can obtain base-\(p\) versions of some games by using \(p\)-saturations.

 \begin{example}[\hspace{0.01em}\cite{irie-pSaturations-2018a}]
 \comment{Exm. [\hspace{0.01em}\cite{irie-pSaturations-2018a}]}
\label{sec:org5907946}
\label{orgaef8186}
Let \(\tilde{\Gamma} = \Gamma(\NN^2, \SatSset{2}{3})\).
Table \ref{org579b3ba} shows the Sprague-Grundy values 
of some positions in \(\tilde{\Gamma}\).
It is easy to see that \(\sg_{\tilde{\Gamma}}(a,0) = \sg_{\tilde{\Gamma}}(0,a) = a\).
Since \((1,1) \in \SatSset{2}{3}\), it follows that
\((0,0)\) is an option of \((1,1)\). Thus \(\sg_{\tilde{\Gamma}}(1,1) = 2\).
We also see that \(\sg_{\tilde{\Gamma}}(1,2) = 0\) because \((0,0)\) 
is not an option of \((1, 2)\).
\begin{table}[H]
\centering
\begin{tabular}{c|cccc}
 & 0 & 1 & 2 & 3\\
\hline
0 & 0 & 1 & 2 & 3\\
1 & 1 & 2 & 0 & 4\\
2 & 2 & 0 & 1 & 5\\
3 & 3 & 4 & 5 & 6\\
\end{tabular}
\caption{Some Sprague-Grundy values in \(\Gamma(\Nim[2], \SatSset{2}{3})\).}
\label{org579b3ba}
\end{table}
\noindent
In general, let \(\tilde{\Gamma}\) be a \(p\)-saturation of Nim \(\Nim[m]\).
If \(\ccA\) is a position in \(\tilde{\Gamma}\), then
\begin{equation}
\label{orgcfffdda}
 \sg_{\tilde{\Gamma}}(\ccA) = \cca^1 \oplus_p \cdots \oplus_p \cca^\ccm.
\end{equation}
In other words, a \(p\)-saturation of Nim is a base-\(p\) version of Nim.
 
\end{example}

\begin{remark}
 \comment{Rem.}
\label{sec:org6ccbd19}
Note that \(\Nim[m]\) is a 2-saturation of itself.
This means that adding an edge \((\ccA, \ccB)\) with \(\ccA - \ccB \in \SatSset{\ccm}{2}\) to \(\Nim[\ccm]\) does not change its Sprague-Grundy function.
Incidentally, it is known that \(\Gamma(\NN^m, \sC)\) is a 2-saturation of \(\Nim[\ccm]\) if and only if
\(\WeightOneSset{\ccm} \subseteq \sC \subseteq \SatSset{\ccm}{2}\) \cite{blass-how-1998}.
 
\end{remark}

\begin{remark}
 \comment{Rem.}
\label{sec:orge1bd709}
\label{org646cf43}
Let \(\ccA\) and \(\ccB\) be two distinct elements of \(\NN^\ccm\) with \(\cca^\cci \ge \ccb^\cci\) for \(\cci \in \set{1, \ldots, \ccm}\).
We show that \(\ccA - \ccB \in \SatSset{\ccm}{\ccp}\) if and only if
\begin{equation}
\label{org0f068fa}
 \bigoplus_\cci \cca^\cci_\ccN \ominus \ccb^\cci_\ccN \neq 0,
\end{equation}
where \(\ccN = \mord_\ccp(\ccA - \ccB)\). Indeed, since \(\cca^\cci \equiv \ccb^\cci \pmod{\ccp^\ccN}\), it follows from Lemma \ref{org912a925} that
\[
  \sum_\cci \cca^\cci - \ccb^\cci \equiv \pexp<\ccp>{0, \ldots, 0, \bigoplus_\cci \cca^\cci_\ccN \ominus \ccb^\cci_\ccN} \pmod{\ccp^{\ccN + 1}}.
\]
Therefore \(\ccA - \ccB \in \SatSset{\ccm}{\ccp}\) if and only if (\ref{org0f068fa}) holds.
 
\end{remark}

 \begin{theorem}[\hspace{0.01em}\cite{irie-pSaturations-2018a}]
 \comment{Thm. [\hspace{0.01em}\cite{irie-pSaturations-2018a}]}
\label{sec:org27dc3e6}
\label{org28c6b37}
Let \(\tilde{\Gamma}\) be a \(p\)-saturation of Welter's game \(\Welter[m]\).
If \(\ccA\) is a position in \(\tilde{\Gamma}\), then
\begin{equation}
\label{orgb76245d}
 \sg_{\tilde{\Gamma}}(\ccA) = \cca^1 \oplus_p \cdots \oplus_p \cca^\ccm \oplus_p \left(\bigopluspm[i < j][][\ccp][-0.3em][0.3em] \ccp^{\ord_\ccp(\cca^i - \cca^j) + 1}  - 1 \right).
\end{equation}
In particular, \(\Welter[m]\) is a 2-saturation of itself.
 
\end{theorem}

 \begin{example}
 \comment{Exm.}
\label{sec:orgbb00ce5}
\label{org6b1ef3c}
Let \(\tilde{\Gamma}\) be a \(5\)-saturation of \(\Welter[3]\) and
\(\ccA\) be the position (7, 5, 3) in \(\tilde{\Gamma}\).
It follows from Theorem \ref{org28c6b37} that
\begin{align*}
 \sg_{\tilde{\Gamma}}(\ccA) &= 7 \oplus_5 5 \oplus_5 3 \oplus_5 (5^{\ord_5(7 - 5) + 1} - 1) \oplus_5 (5^{\ord_5(7 - 3) + 1} - 1) \oplus_5 (5^{\ord_5(5 - 3) + 1} - 1) \\
 &= 10 \oplus_5 4 \oplus_5 4 \oplus_5 4 = 12.
\end{align*}
 
\end{example}

\subsection{\(p\)-Calm subtraction games}
\label{sec:org96eee00}
\label{org04afeb2}
Let \(\Gamma^1\) be a subtraction game satisfying (*).
The next lemma gives a necessary condition for \(\Gamma^1\) to satisfy (PN) when \(\Gamma^2 = \Nim[1]\).
We will show that this condition is sufficient in Section \ref{org61d2a73}.

 \begin{lemma}
 \comment{Lem.}
\label{sec:org521aa65}
\label{orgf5a2cb9}
Let \(\Gamma^1\) be a subtraction game \(\Gamma(\Position, \sC)\) with \(\sC \subseteq \SatSset{\ccm}{\ccp}\),
and let \(\Gamma = \Gamma^1 + \Nim[1]\) and \(\tilde{\Gamma}\) be a \(p\)-saturation of \(\Gamma\).
Suppose that
\begin{equation}
\label{orga871ec5}
 \sg_{\tilde{\Gamma}}(\ccA, \cca) = \sg_{\tilde{\Gamma}^1} (\ccA) \oplus_\ccp a
\end{equation}
for every position \((\ccA, \cca)\) in \(\tilde{\Gamma}\).
If \(\ccB\) is a proper descendant of \(\ccA\) in \(\Gamma^1\), then
\begin{equation}
\label{orgf9cb804}
 \sg_{\tilde{\Gamma}^1}(\ccA) - \sg_{\tilde{\Gamma}^1}(\ccB) \equiv \sum_{i} \cca^i - \ccb^i \pmod{\ccp^{\ccN + 1}},
\end{equation}
where \(\ccN = \mord_p(\ccA - \ccB)\) and \(\tilde{\Gamma}^1\) is a \(p\)-saturation of \(\Gamma^1\).
 
\end{lemma}

\begin{proof}
 \comment{Proof.}
\label{sec:orgfa9e2f2}
We may assume that \(\tilde{\Gamma} = \Gamma(\Position \times \NN, \SatSset{\ccm + 1}{\ccp})\) and
\(\tilde{\Gamma}^1 = \Gamma(\Position, \SatSset{\ccm}{\ccp})\) since \(\sC \subseteq \SatSset{\ccm}{\ccp}\).
Suppose that there are a position \(\ccA\) and its proper descendant \(\ccB\) not satisfying (\ref{orgf9cb804}), that is,
\begin{equation}
\label{orgbf36d76}
 \alpha - \beta \not \equiv \sum_{\cci} \cca^i - \ccb^i \pmod{\ccp^{\ccN + 1}},
\end{equation}
where \(\alpha = \sg_{\tilde{\Gamma}^1}(A)\) and \(\beta = \sg_{\tilde{\Gamma}^1}(B)\).
Note that \(\alpha \neq \beta\) because if \(\alpha = \beta\), then, by Lemma \ref{org912a925},
\[
 \pexp<\ccp>{0, \ldots, 0, \bigoplus_\cci \cca^i_{\ccN} \ominus \ccb^i_{\ccN}} \equiv \sum_{\cci} \cca^i - \ccb^i \not \equiv \alpha - \beta \equiv 0 \pmod{\ccp^{\ccN + 1}},
\]
so \(\ccB\) is an option of \(\ccA\), which is impossible.

We prove that (\ref{orga871ec5}) does not hold.
Consider the two positions \((A, \beta \ominus_p \alpha)\) and \((B, 0)\).
Note that
\[
 \sg_{\tilde{\Gamma}^1}(A) \oplus_p (\beta \ominus_p \alpha) = \beta\quad \tand \quad \sg_{\tilde{\Gamma}^1}(B) \oplus_p 0 = \beta.
\]
We show that \((B, 0)\) is an option of \((A, \beta \ominus_p \alpha)\) in \(\tilde{\Gamma}\), which will imply that (\ref{orga871ec5}) does not hold. 
Let \(\ccM = \ord_\ccp(\beta \ominus_{\ccp} \alpha)\).

 \begin{mycase}[\(\ccM < \ccN\)]
 \comment{Case. [\(\ccM < \ccN\)]}
\label{sec:org4aa460f}
We see that
\[
 \mord_{\ccp}((\ccA, \beta \ominus_{\ccp} \alpha) - (\ccB, 0)) = \min \set{\mord_\ccp(\ccA - \ccB), \ord_{\ccp}(\beta \ominus_{\ccp} \alpha)} = \min\set{\ccN, \ccM} = \ccM.
\]
Since \(a^i \equiv b^i \equiv 0 \pmod{\ccp^\ccN}\), it follows that
\[
 \biggl(\Bigl (\sum_{\cci} a^i - b^i \Bigr) + (\beta \ominus_{\ccp} \alpha - 0) \biggr)_{\ccM} = \beta_\ccM \ominus \alpha_\ccM \neq 0.
\]
Therefore \((B, 0)\) is an option of \((A, \beta \ominus_p \alpha)\) in \(\tilde{\Gamma}\).
 
\end{mycase}

 \begin{mycase}[\(\ccM \ge \ccN\)]
 \comment{Case. [\(\ccM \ge \ccN\)]}
\label{sec:orga5a7207}
Note that
\[
  \mord_{\ccp}((\ccA, \beta \ominus_{\ccp} \alpha) - (\ccB, 0)) = \min \set{\ccN, \ccM} = \ccN.
\]
By Lemma \ref{org912a925},
\begin{equation}
\label{org68f5fdb}
 \sum_\cci \cca^i - \ccb^i \equiv \pexp<\ccp>{0, \ldots, 0, \bigoplus_\cci \cca^i_\ccN \ominus \ccb^i_\ccN} \pmod{\ccp^{\ccN + 1}}
\end{equation}
and 
\begin{equation}
\label{org8fc12f7}
 \alpha - \beta \equiv \pexp<\ccp>{0, \ldots, 0, \alpha_\ccN \ominus \beta_\ccN} \pmod{\ccp^{\ccN + 1}}.
\end{equation}
By combining (\ref{orgbf36d76})--(\ref{org8fc12f7}),
\[
 \alpha_\ccN \ominus \beta_{\ccN} \neq \bigoplus_\cci \cca^i_\ccN \ominus \ccb^i_\ccN.
\]
Hence
\[
 \Bigl(\bigoplus_\cci \cca^i_\ccN \ominus \ccb^i_\ccN  \Bigr) \oplus (\beta_\ccN \ominus \alpha_\ccN) \neq 0.
\]
This implies that \((\ccB, 0)\) is an option of \((\ccA, \beta \ominus_\ccp \alpha)\) in \(\tilde{\Gamma}\). Therefore (\ref{orga871ec5}) does not hold.
 
\end{mycase} 
\end{proof}

\comment{connect}
\label{sec:org8bbeb24}
Let \(\Gamma\) be a subtraction game \(\Gamma(\cP, \sC)\) with \(\sC \subseteq \SatSset{\ccm}{\ccp}\).
The game \(\Gamma\) is said to be \emph{\(p\)-calm} if it satisfies (\ref{orgf9cb804}),
that is, for every position \(\ccA\) and every proper descendant \(\ccB\) of \(\ccA\),
\[
 \sg_{\tilde{\Gamma}}(\ccA) - \sg_{\tilde{\Gamma}}(\ccB) \equiv \sum_{i} \cca^i - \ccb^i \pmod{\ccp^{\ccN + 1}},
\]
where \(\tilde{\Gamma}\) is a \(p\)-saturation of \(\Gamma\) and \(\ccN = \mord_\ccp(\ccA - \ccB)\).
For example, \(\Nim[1]\) is \(p\)-calm.
Indeed, because \(\WeightOneSset{1} = \SatSset{1}{\ccp}\), we see that \(\Nim[1]\) is a \(p\)-saturation of itself. 
Since \(\sg_{\Nim[1]}(\cca) = \cca\) for \(\cca \in \NN\), it follows that \(\Nim[1]\) is \(p\)-calm.

\begin{remark}
 \comment{Rem.}
\label{sec:org09480f8}
\label{org1b0b16b}
There exist non-\(p\)-calm subtraction games.
For example, let \(\Gamma\) be the subtraction game \(\Gamma(\set{0, \ccp}, \WeightOneSset{1})\).
It is clear that \(\Gamma\) is a \(p\)-saturation of itself.
Since \(\sg_{\Gamma}(0) = \sg_{\Gamma}(\pexp<\ccp>{0}) = 0\) and \(\sg_{\Gamma}(\ccp) = \sg_{\Gamma}(\pexp<\ccp>{0,1}) = 1\), it follows that \(\Gamma\) is not \(p\)-calm.
 
\end{remark}

\subsection{A base-\(p\) Sprague-Grundy type theorem}
\label{sec:org36a3c62}
\label{org61d2a73}
The next theorem says that \(p\)-calm subtraction games satisfy (PN) and are closed under disjunctive sum.

 \begin{theorem}
 \comment{Thm.}
\label{sec:orgb694be2}
\label{org25184ea}
For \(i \in \set{1, \ldots, \cck}\), let \(\Gamma^i\) be a \(p\)-calm subtraction game.
Then the disjunctive sum \(\Gamma^1 + \cdots + \Gamma^\cck\) is \(p\)-calm.
Moreover, if \(\tilde{\Gamma}\) is a p-saturation of \(\Gamma^1 + \cdots + \Gamma^\cck\) and \(\bbA\) is a position in \(\tilde{\Gamma}\), then 
\[
 \sg_{\tilde{\Gamma}}(\bbA) = \sg_{\tilde{\Gamma}^1}(\ccA^1) \oplus_p \cdots \oplus_p \sg_{\tilde{\Gamma}^\cck}(\ccA^\cck),
\]
where \(\bbA = (\ccA^1, \ldots, \ccA^\cck)\) and \(\tilde{\Gamma}^i\) is a \(p\)-saturation of \(\Gamma^i\).
 
\end{theorem}

\comment{connect}
\label{sec:org2e95032}
To prove Theorem \ref{org25184ea}, we use the following simple lemma.

 \begin{lemma}
 \comment{Lem.}
\label{sec:org2322d02}
\label{org2c3e6de}
Let \(\Gamma\) be a \(p\)-calm subtraction game and \(\phi\) be the Sprague-Grundy function of its \(p\)-saturation.
If \(\ccA\) is a position in \(\Gamma\) and \(\ccB\) is its proper descendant, then
\begin{equation}
\label{orgbba0d43}
 \phi(\ccA) \ominus_\ccp \phi(\ccB) \equiv \phi(\ccA) - \phi(\ccB) \equiv  \pexp<\ccp>{0, \ldots, 0, \bigoplus_\cci \cca^\cci_\ccN \ominus \ccb^\cci_\ccN} \pmod{\ccp^{\ccN + 1}},
\end{equation}
where \(\ccN = \mord_\ccp(\ccA - \ccB)\).
 
\end{lemma}

\begin{proof}
 \comment{Proof.}
\label{sec:orgfb286f3}
By Lemma \ref{org912a925},
\[
 \sum_\cci \cca^\cci - \ccb^\cci \equiv \bigopluspm[\cci][][\ccp][-0.35em][0.35em] \cca^\cci \ominus_\ccp \ccb^\cci \equiv \pexp<\ccp>{0, \ldots, 0, \bigoplus_\cci \cca^\cci_\ccN \ominus \ccb^\cci_\ccN} \pmod{\ccp^{\ccN + 1}}.
\]
Since \(\Gamma\) is \(p\)-calm, it follows that
\[
 \sum_\cci \cca^i - \ccb^i \equiv \phi(\ccA) - \phi(\ccB) \pmod{\ccp^{\ccN + 1}}.
\]
We show that \(\phi(\ccA) - \phi(\ccB) \equiv \phi(\ccA) \ominus_\ccp \phi(\ccB) \pmod{\ccp^{\ccN + 1}}\). 
Since 
\[
 \phi(\ccA) - \phi(\ccB) \equiv \sum_\cci \cca^i - \ccb^i \equiv 0 \pmod{\ccp^\ccN},
\]
we see that \(\phi(\ccA) \equiv \phi(\ccB) \pmod{\ccp^\ccN}\).
It follows from Lemma \ref{org912a925} that 
\[
  \phi(\ccA) - \phi(\ccB) \equiv \phi(\ccA) \ominus_\ccp \phi(\ccB) \pmod{\ccp^{\ccN + 1}}.
\]
Therefore (\ref{orgbba0d43}) holds.
\end{proof}

\begin{proof}[\bfseries proof of Theorem \ref{org25184ea}]
 \comment{Proof. [\bfseries proof of Theorem \ref{org25184ea}]}
\label{sec:org3cf56fe}
\rmfamily
For \(\cci \in \set{1, \ldots, \cck}\), we may assume that 
\(\tilde{\Gamma}^i = \Gamma(\cP^\cci, \SatSset{\ccm^\cci}{p})\) 
and \(\tilde{\Gamma} = \Gamma(\cP^1 \times \cdots \times \cP^k, \SatSset{\ccm}{\ccp})\),
where \(\ccm = \ccm^1 + \cdots + \ccm^\cck\). Let \(\phi^i(\ccA^i) = \sg_{\tilde{\Gamma}^i}(A^i)\) and
\(\phi(\bbA) = \phi^1(\ccA^1) \oplus_p \cdots \oplus_p \phi^\cck(\ccA^\cck)\).
To prove that \(\sg_{\tilde{\Gamma}}(\bbA) = \phi(\bbA)\), it suffices to show the following two statements.

\begin{description}
\item[{(SG1)}] If \(\bbB\) is an option of \(\bbA\) in \(\tilde{\Gamma}\), then \(\phi(\bbB) \neq \phi(\bbA)\).
\item[{(SG2)}] If \(0 \le \beta < \phi(\bbA)\), then \(\phi(\bbB) = \beta\) for some option \(\bbB\) of \(\bbA\) in \(\tilde{\Gamma}\).
\end{description}

Let \(\bbB\) be a proper descendant \((\ccB^1, \ldots, \ccB^\cck)\) of \(\bbA\) and let
\[
  \ccN = \mord_p(\bbA - \bbB) = \min \set{\ord_p(\cca^{\cci, \ccj} - \ccb^{\cci, \ccj}) : 1 \le \cci \le \cck, 1 \le \ccj \le \ccm^i},
\]
where \(\ccA^i = (\cca^{i, 1}, \ldots, \cca^{i, \ccm^i})\) and \(\ccB^i = (\ccb^{i, 1}, \ldots, \ccb^{i, \ccm^i})\).
We first show that
\begin{equation}
\label{orgb68d3f9}
 \phi(\bbA) - \phi(\bbB) \equiv \sum_{i,j} \cca^{i,j} - \ccb^{i,j} \equiv \pexp<\ccp>{0, \ldots, 0, \bigoplus_{i,j} a^{i,j}_\ccN \ominus b^{i,j}_{\ccN}} \pmod{\ccp^{\ccN + 1}}.
\end{equation}
Since \(\cca^{i, j} \equiv \ccb^{i, j} \pmod{\ccp^\ccN}\),
it follows from Lemma \ref{org912a925} that
\[
 \sum_{\cci, \ccj} \cca^{i, j} - \ccb^{i, j} \equiv \pexp<\ccp>{0, \ldots, 0, \bigoplus_{i, j} \cca^{i,j}_\ccN \ominus \ccb^{i, j}_\ccN} \pmod{\ccp^{\ccN + 1}}.
\]
Now, since \(\Gamma^i\) is \(p\)-calm, it follows from Lemma \ref{org2c3e6de} that
\begin{equation}
\label{orgd1ba4c5}
 \phi^i(\ccA^i) - \phi^i(\ccB^i) \equiv \phi^i(\ccA^i) \ominus_{\ccp} \phi^i(\ccB^i) \equiv \pexp<\ccp>{0, \ldots, 0, \bigoplus_{j} a^{i,j}_\ccN \ominus b^{i,j}_{\ccN}} \pmod{\ccp^{\ccN + 1}}.
\end{equation}
Moreover, since \(\phi^i(\ccA^i) \equiv \phi^i(\ccB^i) \pmod{\ccp^\ccN}\), we see that
\(\phi(\bbA) \equiv \phi(\bbB) \pmod{\ccp^\ccN}\). By Lemma \ref{org912a925} and (\ref{orgd1ba4c5}),
\begin{align*}
 \phi(\bbA) - \phi(\bbB) &\equiv \phi(\bbA) \ominus_\ccp \phi(\bbB) \\
 &\equiv \bigoplus_\cci \phi^i(\ccA^i) \ominus_{\ccp} \phi^i(\ccB^i) \equiv \pexp<\ccp>{0, \ldots, 0, \bigoplus_{i,j} a^{i,j}_\ccN \ominus b^{i,j}_{\ccN}} \pmod{\ccp^{\ccN + 1}}.
\end{align*}
Therefore (\ref{orgb68d3f9}) holds.

We now show (SG1).
Let \(\bbB\) be an option of \(\bbA\) and
\(\ccN = \mord_p(\bbA - \bbB)\).
Then 
\[
 \bigoplus_{\cci, \ccj} \cca^{\cci, \ccj}_\ccN \ominus \ccb^{\cci, \ccj}_\ccN \neq 0.
\]
By (\ref{orgb68d3f9}), (SG1) holds.

We next show (SG2).
Let \(\alpha^\cci = \phi^i(\ccA^i)\) and \(\alpha = \alpha^1 \oplus_p \cdots \oplus_p \alpha^\cck\) (\(= \phi(\bbA)\)).
Let \(\beta\) be an integer with \(0 \le \beta < \alpha\).
We first construct a descendant \(\bbB\) of \(\bbA\) with \(\phi(\bbB) = \beta\).
When we consider \((\alpha^1, \ldots, \alpha^\cck)\) as a position in a \(p\)-saturation of Nim,
its Sprague-Grundy value is equal to \(\alpha\) as we have mentioned in Example \ref{orgaef8186}.
Therefore there exist \(\beta^1, \ldots, \beta^k \in \NN\) satisfying the following three conditions:
\begin{enumerate}
\item \(\beta^\cci \le \alpha^\cci\).
\item \(\ord_p\left(\sum_\cci \alpha^\cci - \beta^\cci\right) = \min \set{\ord_p(\alpha^\cci - \beta^\cci) : 1 \le \cci \le \cck}\).
\item \(\beta^1 \oplus_p \cdots \oplus_p \beta^k = \beta\).
\end{enumerate}
\noindent
If \(\beta^\cci = \alpha^\cci\), then let \(\ccB^\cci = \ccA^\cci\).
If \(\beta^\cci < \alpha^\cci\), then, since \(\alpha^\cci = \phi^i(\ccA^i) = \sg_{\tilde{\Gamma}^i}(\ccA^i)\),
we see that
\(\ccA^i\) has an option \(\ccB^\cci\) such that \(\phi^i(\ccB^\cci) = \beta^\cci\) in \(\tilde{\Gamma}^i\).
Let \(\bbB = (\ccB^1, \ldots, \ccB^\cck)\). Then \(\phi(\bbB) = \beta^1 \oplus_p \cdots \oplus_p \beta^\cck = \beta\). 

We prove that \(\bbB\) is an option of \(\bbA\).
By (\ref{orgb68d3f9}), it suffices to show that \(\beta_\ccN \neq \alpha_\ccN\), where \(\ccN = \mord_p(\bbA - \bbB)\).
Let \(\ccN^i = \mord_p(\ccA^\cci - \ccB^\cci)\).
We first show that
\begin{equation}
\label{org2204764}
  \ccN^i = \ord_p (\alpha^\cci - \beta^\cci).
\end{equation}
Indeed, if \(\alpha^\cci = \beta^\cci\), then \(\ccA^\cci = \ccB^\cci\), so (\ref{org2204764}) holds.
Suppose that \(\alpha^\cci > \beta^\cci\).
Since \(\ccB^i\) is an option of \(\ccA^i\) and \(\Gamma^i\) is \(p\)-calm, it follows from Lemma \ref{org2c3e6de} that
\begin{align*}
 \ccN^i &= \ord_p \Big(\sum_j \cca^{i, j} - \ccb^{i, j }\Big) \\
 &= \ord_p \Big(\phi^i(\ccA^i) - \phi^i(\ccB^i) \Big) = \ord_p(\alpha^\cci - \beta^\cci).
\end{align*}
Hence (\ref{org2204764}) holds.
We now prove that \(\beta_\ccN \neq \alpha_\ccN\).
By (\ref{org2204764}) and (ii),
\begin{equation}
\label{org24d7a3b}
\begin{split}
 \ccN = \min \Set{\ccN^\cci : 1 \le \cci \le \cck} &= \min \Set{\ord_p(\alpha^\cci - \beta^\cci) : 1 \le \cci \le \cck} \\
 &= \ord_p\Big(\sum_\cci \alpha^\cci - \beta^\cci\Big).
\end{split}
\end{equation}
Now, since \(\ord_\ccp(\alpha^\cci - \beta^\cci) = \ccN^\cci \ge \ccN\),
we see that \(\alpha^\cci \equiv \beta^\cci \pmod{\ccp^\ccN}\).
It follows from Lemma \ref{org912a925} that
\[
 \Big(\sum_{\cci} \alpha^\cci - \beta^\cci \Big)_\ccN  = \bigoplus_{\cci} \alpha^i_N \ominus \beta^i_N   = \alpha_N \ominus \beta_N.
\]
By (\ref{org24d7a3b}), we see that \(\alpha_N \neq \beta_N\), and so \(\bbB\) is an option of \(\bbA\) in \(\tilde{\Gamma}\).
Therefore \(\phi(\bbA) = \sg_{\tilde{\Gamma}}(\bbA)\). In particular, 
if \(\bbA\) is a position and \(\bbB\) is its proper descendant, then, by (\ref{orgb68d3f9}),
\[
 \sg_{\tilde{\Gamma}}(\bbA)  - \sg_{\tilde{\Gamma}}(\bbB) = \phi(\bbA) - \phi(\bbB) \equiv \sum_{i,j} \cca^{i,j} - \ccb^{i,j} \pmod{\ccp^{\ccN + 1}},
\]
where \(\ccN = \mord_\ccp(\bbA - \bbB)\).
Hence \(\Gamma^1 + \cdots + \Gamma^\cck\) is \(p\)-calm.
\end{proof}

 \begin{corollary}
 \comment{Cor.}
\label{sec:org5505439}
\label{orgbaecf19}
If \(\Gamma\) is a subtraction game \(\Gamma(\cP, \sC)\) with \(\sC \subseteq \SatSset{\ccm}{\ccp}\),
then \(\Gamma\) and \(\Nim[1]\) satisfy \textup{(PN)}
if and only if \(\Gamma\) is \(p\)-calm.
 
\end{corollary}

\begin{remark}
 \comment{Rem.}
\label{sec:org27d2438}
\label{org0eabe86}
We can generalize Corollary \ref{orgbaecf19} as follows.
Let \(\Gamma^1\) be a subtraction game \(\Gamma(\cP, \sC)\) with \(\sC \subseteq \SatSset{\ccm}{\ccp}\),
and let \(\Gamma^2\) be a \(p\)-calm subtraction game and
\(\tilde{\Gamma}^2\) be its \(p\)-saturation.
Suppose that there exists a position \(\ccA^2 \in \Position[\tilde{\Gamma}^2]\) such that
\(\sg_{\tilde{\Gamma}^2}(\ccA^2) = \alpha\) for every \(\alpha \in \NN\). 
Then \(\Gamma^1\) and \(\Gamma^2\) satisfy (PN)
if and only if \(\Gamma^1\) is \(p\)-calm.
We can prove this by the same argument as in the proof of Lemma \ref{orgf5a2cb9},
so we only sketch it.
Suppose that \(\Gamma^1\) is not \(p\)-calm.
Then there exist a position \(\ccA^1\) and its proper descendant \(\ccB^1\) 
such that 
\begin{equation}
\label{org394aa98}
 \alpha - \beta \not \equiv \sum_{\ccj} \cca^{1, \ccj} - \ccb^{1, \ccj} \pmod{\ccp^{\ccN + 1}},
\end{equation}
where \(\alpha = \sg_{\tilde{\Gamma}^1}(\ccA^1)\), \(\beta = \sg_{\tilde{\Gamma}^1}(\ccB^1)\), \(\tilde{\Gamma}^1\) is a \(p\)-saturation of \(\Gamma^1\),
and \(\ccN = \mord_\ccp(\ccA^1 - \ccB^1)\).
By assumption, \(\tilde{\Gamma}^2\) has a position \(\ccA^2\) such that
\(\sg_{\tilde{\Gamma}^2}(\ccA^2) = \beta \ominus_p \alpha (> 0)\).
This position \(\ccA^2\) has an option \(\ccB^2\) with \(\sg_{\tilde{\Gamma}^2}(\ccB^2) = 0\).
Let \(\bbA = (\ccA^1, \ccA^2)\) and \(\bbB = (\ccB^1, \ccB^2)\).
It is sufficient to show that \(\bbB\) is an option of \(\bbA\) in \(\tilde{\Gamma}\),
where \(\tilde{\Gamma} = \Gamma(\Position[\Gamma^1] \times \Position[\Gamma^2], \SatSset{\ccn}{\ccp})\).
Let \(\ccM = \mord_{\ccp}(\ccA^2 - \ccB^2)\) (\(= \ord_{\ccp}(\sum_{\ccj} \cca^{2, \ccj} - \ccb^{2, \ccj})\)).
Since \(\Gamma^2\) is \(p\)-calm, it follows from Lemma \ref{org2c3e6de} that
\begin{equation}
\label{org13d25f1}
 \sum_{\ccj} \cca^{2, \ccj} - \ccb^{2, \ccj} \equiv \beta \ominus_{\ccp} \alpha \equiv \pexp<\ccp>{0,\ldots,0, \beta_\ccM \ominus \alpha_\ccM} \pmod{\ccp^{\ccM + 1}}.
\end{equation}
Suppose that \(\ccM < \ccN\). Then \(\mord_{\ccp}(\bbA - \bbB) = \ccM\).
Since \((\sum \cca^{i, j} - \ccb^{i, j})_\ccM  = (\sum \cca^{2, j} - \ccb^{2, j})_{\ccM} \neq 0\), 
it follows that \(\bbB\) is an option of \(\bbA\).
Suppose that \(\ccM \ge \ccN\).
Then \(\mord_{\ccp}(\bbA - \bbB) = \ccN\).
It follows from (\ref{org394aa98}) and (\ref{org13d25f1}) that \(\alpha_\ccN \ominus \beta_\ccN \neq \bigoplus_{\ccj} \cca^{1, \ccj}_\ccN \ominus \ccb^{1, \ccj}_\ccN\).
Hence 
\[
 \Bigl(\sum_{\cci, \ccj} \cca^{i,j} - \ccb^{i,j}\Bigr)_\ccN = \Bigl(\bigoplus_\ccj \cca^{1, \ccj}_\ccN \ominus \ccb^{1, \ccj}_\ccN \Bigr) \oplus \beta_\ccN \ominus \alpha_\ccN \neq 0.
\]
This implies that \(\bbB\) is an option of \(\bbA\).
Therefore \(\Gamma^1\) and \(\Gamma^2\) do not satisfy (PN).
 
\end{remark}

\subsection{The \(p\)-calmness of Welter's game}
\label{sec:orgd054ab0}
For a position \(\ccA\) in Welter's game,
let \(\ppsi<p>(\ccA)\) denote the right-hand side of (\ref{orgb76245d}) in Theorem \ref{org28c6b37},
that is,
\begin{equation}
\label{orgd46f61a}
 \ppsi<\ccp>(\ccA) = \cca^1 \oplus_p \cdots \oplus_p \cca^\ccm \oplus_p \Bigl(\bigopluspm[i < j][][\ccp][-0.3em][0.3em] \ccp^{\ord_\ccp(\cca^i - \cca^j) + 1}  - 1 \Bigr).
\end{equation}

 \begin{theorem}
 \comment{Thm.}
\label{sec:orgf2afbc4}
\label{org47862fb}
Welter's game \(\Welter[\ccm]\) is \(p\)-calm.
In particular, if \(\tilde{\Gamma}\) is a \(p\)-saturation of \(\Welter[\ccm^1] + \cdots + \Welter[\ccm^\cck]\) and \(\bbA\) is a position in \(\tilde{\Gamma}\), then
\[
 \sg_{\tilde{\Gamma}}(\bbA) = \ppsi<\ccp>(\ccA^1) \oplus_\ccp \cdots \oplus_\ccp \ppsi<\ccp>(\ccA^\cck),
\]
where \(\bbA = (\ccA^1, \ldots, \ccA^\cck)\).
 
\end{theorem}

\begin{proof}
 \comment{Proof.}
\label{sec:orgad08dac}
Let \(\ccA\) be a position in \(\Welter[\ccm]\) and
\(\ccB\) be its proper descendant. Set \(\ccN = \mord_p(\ccA - \ccB)\).
Then \(\cca^{\cci} \equiv \ccb^{\cci} \pmod{\ccp^\ccN}\). In particular,
\(\cca^{\cci} - \cca^{\ccj} \equiv \ccb^{\cci} - \ccb^{\ccj} \pmod{\ccp^\ccN}\).
Since
\[
 (\ccp^{\ord_{\ccp}(\cch) + 1} - 1)_{\ccL} = \begin{cases}
 \ccp - 1 & \tif \cch \equiv 0 \pmod{\ccp^\ccL}, \\
 0 & \tif \cch \not \equiv 0 \pmod{\ccp^\ccL}, \\
 \end{cases}
\]
we see that \((\ccp^{\ord_{\ccp}(\cca^\cci - \cca^\ccj) + 1} - 1)_\ccL = (\ccp^{\ord_{\ccp}(\ccb^\cci - \ccb^\ccj) + 1} - 1)_{\ccL}\) for \(0 \le \ccL \le \ccN\). Thus
\[
 \ccp^{\ord_{\ccp}(\cca^\cci - \cca^\ccj) + 1} - 1 \equiv \ccp^{\ord_{\ccp}(\ccb^\cci - \ccb^\ccj) + 1} - 1 \pmod{\ccp^{\ccN + 1}}.
\]
It follows from (\ref{orgd46f61a}) that
\[
 \ppsi<\ccp>(\ccA) - \ppsi<\ccp>(\ccB) \equiv 0 \equiv \sum_{i} a^i - b^i \pmod{\ccp^{\ccN}}
\]
and
\[
 (\ppsi<\ccp>(\ccA) - \ppsi<\ccp>(\ccB))_{\ccN} = \bigoplus_{\cci} \cca^i_\ccN \ominus \ccb^i_\ccN = \Bigl(\sum_{i} \cca^i - \ccb^i\Bigr)_N.
\]
Therefore Welter's game is \(p\)-calm.
\end{proof}

 \begin{example}
 \comment{Exm.}
\label{sec:org2b403ec}
\label{org162f640}
Let \(\Gamma\) be a \(5\)-saturation of \(\Welter[3] + \Welter[1]\) and
\(\bbA\) be the position \(((7, 5, 3), (3))\) in \(\Gamma\).
By Theorem \ref{org47862fb},
\[
 \sg_{\Gamma}(\bbA) = \ppsi<5>(7,5,3) \oplus_5 \ppsi<5>(3) = 12 \oplus_3 3 = 10.
\]
 
\end{example}

\subsection{Full positions in Welter's game}
\label{sec:org932c3ef}
\label{orgf2856b2}
Recall that Theorem \ref{org28c6b37} provides a formula for the Sprague-Grundy function of a \(p\)-saturation of Welter's game.
This theorem was proved using the next proposition.
In this section, we generalize this proposition to disjunctive sums of Welter's games using Theorem \ref{org47862fb}.

 \begin{proposition}[\hspace{0.01em}\cite{irie-pSaturations-2018a}]
 \comment{Prop. [\hspace{0.01em}\cite{irie-pSaturations-2018a}]}
\label{sec:org22ad4a8}
\label{org8a24e5c}
Every position \(\ccA\) in Welter's game \(\Welter[\ccm]\) has a descendant \(\ccB\) such that \(\lg_{\Welter[\ccm]}(\ccB) = \ppsi<\ccp>(\ccB) = \ppsi<\ccp>(\ccA)\).
 
\end{proposition}
 \begin{example}
 \comment{Exm.}
\label{sec:org68dd2b0}
\label{org9f6734a}
Let \(\ccA\) be the position \((6, 4, 2)\) in \(\Welter[3]\).
Then
\begin{align*}
 \ppsi<2>(\ccA) &= 6 \oplus_2 4 \oplus_2 2 \oplus_2 (2^{\ord_2(6 - 4) + 1} - 1) \oplus_2 (2^{\ord_2(6 - 2) + 1} - 1) \oplus_2 (2^{\ord_2(4 - 2) + 1} - 1) \\
 &= 0 \oplus_2 3 \oplus_2 7 \oplus_2 3 = 7.
\end{align*}
Proposition \ref{org8a24e5c} says that \(\ccA\) has a descendant \(\ccB\)
such that \(\lg_{\Welter[3]}(\ccB) = \ppsi<2>(\ccB) = 7\).
Indeed, if \(\ccB = (5, 3, 2)\), then 
\[
 \lg_{\Welter[3]}(\ccB) = 5 + 3 + 2 - {3 \choose 2} = 7
\]
and
\begin{align*}
 \ppsi<2>(\ccB) &= 5 \oplus_2 3 \oplus_2 2 \oplus_2 (2^{\ord_2(5 - 3) + 1} - 1) \oplus_2 (2^{\ord_2(5 - 2) + 1} - 1) \oplus_2 (2^{\ord_2(3 - 2) + 1} - 1) \\
 &= 4 \oplus_2 3 \oplus_2 1 \oplus_2 1 = 7.
\end{align*}
 
\end{example}

\begin{remark}
 \comment{Rem.}
\label{sec:org2986881}
\label{org4706195}
In general, if \(\ccA\) is a position in an impartial game \(\Gamma\), then
\begin{equation}
\label{orgec20038}
 \sg_{\Gamma}(\ccA) \le \lg_{\Gamma}(\ccA).
\end{equation}
If \(\sg_{\Gamma}(\ccA) = \lg_{\Gamma}(\ccA)\), then \(\ccA\) is said to be \emph{full}.
Consider the following condition on an impartial game \(\Gamma\):

\begin{description}
\item[{(FD)}] Every position \(\ccA\) in \(\Gamma\) has a full descendant \(\ccB\) with the same Sprague-Grundy value as \(\ccA\).
\end{description}

\noindent
It is easy to show that a \(p\)-saturation of Nim satisfies the condition (FD).
It follows from Proposition \ref{org8a24e5c} and Theorem \ref{org28c6b37} that 
a \(p\)-saturation of Welter's game also satisfies (FD).
Our aim is to prove that a \(p\)-saturation of disjunctive sums of Welter's games satisfies (FD).
To this end, we show two lemmas.
 
\end{remark}

 \begin{lemma}
 \comment{Lem.}
\label{sec:org6281f17}
\label{org21901d4}
Let \(\ccA\) be a full position in an impartial game \(\Gamma\).
If \(\lg_{\Gamma}(\ccA) > 0\), then \(\ccA\) has a full option \(\ccB\) with \(\lg_{\Gamma}(\ccB) = \lg_{\Gamma}(\ccA) - 1\).
In particular, if \(0 \le \beta \le \lg_{\Gamma}(\ccA) - 1\), then \(\ccA\) has a full descendant \(\ccB\) with \(\lg_\Gamma(\ccB) = \beta\).
 
\end{lemma}

\begin{proof}
 \comment{Proof.}
\label{sec:orga2159d2}
Since \(\ccA\) is full, it follows that \(\sg_{\Gamma}(\ccA) = \lg_{\Gamma}(\ccA) > 0\).
Hence \(\ccA\) has an option \(\ccB\) with \(\sg_{\Gamma}(\ccB) = \sg_{\Gamma}(\ccA) - 1 = \lg_{\Gamma}(\ccA) - 1\).
Since \(\sg_{\Gamma}(\ccB) \le \lg_{\Gamma}(\ccB) \le \lg_{\Gamma}(\ccA) - 1\), we see that \(\sg_{\Gamma}(\ccB) = \lg_{\Gamma}(\ccB) = \lg_{\Gamma}(\ccA) - 1\).
Thus \(\ccB\) is the desired option of \(\ccA\).
\end{proof}

 \begin{lemma}
 \comment{Lem.}
\label{sec:org6b4e570}
\label{org6895abb}
For \(\cci \in \set{1, \ldots, \cck}\), let \(\Gamma^i\) be an impartial game satisfying \textup{(FD)}, 
and let \(\Gamma\) be an impartial game with position set \(\Position[\Gamma^1] \times \cdots \times \Position[\Gamma^\cck]\).
If, for \(\bbA = (\ccA^1, \ldots, \ccA^\cck) \in \Position[\Gamma]\),
\begin{equation}
\label{org924b465}
 \sg_{\Gamma}(\bbA) = \sg_{\Gamma^1}(\ccA^1) \oplus_\ccp \cdots \oplus_\ccp \sg_{\Gamma^\cck}(\ccA^\cck)
\end{equation}
and
\begin{equation}
\label{org28b7cc2}
 \lg_{\Gamma}(\bbA) = \lg_{\Gamma^1}(\ccA^1) + \cdots + \lg_{\Gamma^\cck}(\ccA^\cck),
\end{equation}
then \(\Gamma\) satisfies \textup{(FD)}.
 
\end{lemma}

\begin{proof}
 \comment{Proof.}
\label{sec:org689e0c0}
Let \(\bbA\) be a position \((\ccA^1, \ldots, \ccA^\cck)\) in \(\Gamma\),
and let \(\alpha^i = \sg_{\Gamma^{\cci}}(\ccA^\cci)\) and \(\alpha = \sg_{\Gamma}(\bbA) = \alpha^1 \oplus_p \cdots \oplus_p \alpha^\cck\).
By Lemma \ref{org21901d4}, it suffices to show that 
\(\bbA\) has a full descendant \(\bbB\) such that \(\sg_{\Gamma}(\bbB) \ge \alpha\).
Let
\[
 \ccM = \max \Set{\ccL \in \NN : \alpha^1_L + \cdots + \alpha^\cck_L \ge p},
\]
where \(\max \emptyset = -1\).
Define
\[
 \beta = (p^{\ccM + 1} - 1) + \sum_{\ccL \ge \ccM + 1} \alpha_\ccL p^{\ccL} = \pexp<\ccp>{p - 1, \ldots, p - 1, \alpha_{\ccM + 1}, \alpha_{\ccM + 2}, \ldots}.
\]
Then \(\beta \ge \alpha\). We will show that \(\bbA\) has a full descendant \(\bbB\) such that \(\sg_{\Gamma}(\bbB) = \beta\).

We first show that there exist \(\beta^1, \ldots, \beta^k \in \NN\) satisfying the following two conditions:
\begin{enumerate}
\item \(\beta^i \le \alpha^i\).
\item \(\beta^1 + \cdots + \beta^k = \beta^1 \oplus_p \cdots \oplus_p \beta^k = \beta\).
\end{enumerate}
If \(\ccM = -1\), then \(\alpha^1, \ldots, \alpha^\cck\) satisfy the two conditions.
Suppose that \(\ccM \ge 0\).
For \(\ccL \ge \ccM + 1\), let \(\beta^i_\ccL = \alpha^i_\ccL\).
Since \(\alpha^1_\ccM + \cdots + \alpha^\cck_\ccM \ge p\),
there exist 
\(\beta^1_\ccM, \ldots, \beta^\cck_\ccM\) such that
\[
 \beta^i_\ccM \le \alpha_M^i \quad \tand \quad \beta^1_\ccM + \cdots + \beta^k_\ccM = p - 1 = \beta_\ccM.
\]
By rearranging \(\alpha^\cci\) if necessary,
we may assume that \(\beta^1_{\ccM} < \alpha^1_{\ccM}\).
For \(\ccL \le \ccM - 1\), let \(\beta^1_L = p - 1\) and \(\beta^i_L = 0\) for \(i \ge 2\).
Let \(\beta^i = \pexp<\ccp>{\beta^i_0, \beta^i_1, \ldots}\).
Then \(\beta^1, \ldots, \beta^\cck\) satisfy (i) and (ii).

Since \(\Gamma^i\) satisfies (FD), it follows from Lemma \ref{org21901d4} that \(\ccA^\cci\) has a full descendant \(\ccB^i\) 
such that \(\sg_{\Gamma^\cci}(\ccB^i) = \beta^i\).
Let \(\bbB = (\ccB^1, \ldots, \ccB^k)\).
Then 
\[
 \sg_{\Gamma}(\bbB) = \beta^1 \oplus_p \cdots \oplus_p \beta^k = \beta.
\]
Since 
\[
 \beta = \beta^1 + \cdots + \beta^\cck = \lg_{\Gamma^1}(\ccB^1) + \cdots + \lg_{\Gamma^\cck}(\ccB^\cck) = \lg_{\Gamma}(\ccB),
\]
it follows that \(\bbB\) is the desired descendant of \(\bbA\).
\end{proof}

 \begin{proposition}
 \comment{Prop.}
\label{sec:org24f939c}
\label{org14ff579}
A \(p\)-saturation of a disjunctive sum of Welter's games satisfies \textup{(FD)}.
 
\end{proposition}

\begin{proof}
 \comment{Proof.}
\label{sec:org94701da}
For \(\cci \in \set{1, \ldots, \cck}\), let \(\Gamma^i = \Welter[\ccm^\cci]\) and \(\Gamma = \Gamma^1 + \cdots + \Gamma^\cck\).
Let \(\tilde{\Gamma}\) be a \(p\)-saturation of \(\Gamma\).
It is obvious that \(\lg_{\tilde{\Gamma}}(\bbA) = \lg_{\tilde{\Gamma}^1}(\ccA^1) + \cdots + \lg_{\tilde{\Gamma}^k}(\ccA^k)\),
where \(\tilde{\Gamma}^i\) is a \(p\)-saturation of \(\Gamma^i\).
Moreover, it follows from Theorem \ref{org47862fb} that
\(\sg_{\tilde{\Gamma}}(\bbA) = \sg_{\tilde{\Gamma}^1}(\ccA^1) \oplus_p \cdots \oplus_p \sg_{\tilde{\Gamma}^k}(\ccA^k)\).
Since \(\tilde{\Gamma}^1, \ldots, \tilde{\Gamma}^\cck\) satisfy (FD),
it follows from Lemma \ref{org6895abb} that so does \(\tilde{\Gamma}\).
\end{proof}

 \begin{example}
 \comment{Exm.}
\label{sec:org6cc8150}
\label{org52d18b3}
Let \(\Gamma = \Welter[3] + \Welter[2]\)
and \(\bbA\) be the position \(((6,5,2), (3,1))\) in \(\Gamma\).
Note that \(\Gamma\) is a 2-saturation of itself and
\[
 \sg_{\Gamma}(\bbA) = \ppsi<2>(6, 5, 2) \oplus_2 \ppsi<2>(3, 1) = 6 \oplus_2 1 = 7.
\]
By Proposition \ref{org14ff579}, the position \(\bbA\) has a full descendant \(\bbB\)
such that \(\lg_{\Gamma}(\bbB) = \ppsi<2>(\bbB) = 7\).
Indeed, if \(\bbB = ((4,3,1), (3,0))\), then
\[
 \sg_{\Gamma}(\bbB) = \ppsi<2>(4, 3, 1) \oplus_2 \ppsi<2>(3,0) = 5 \oplus_2 2 =  7
\]
and
\[
 \lg_\Gamma(\bbB) = \lg_{\Welter[3]}(4,3,1) + \lg_{\Welter[2]}(3,0) = 5 + 2 = 7.
\]
 
\end{example}

\section{The \(p'\)-Component Theorem}
\label{sec:org3ab6cdc}
\label{org3ea736a}
By rewriting Proposition \ref{org14ff579} in terms of Young diagrams, we generalize the \(p'\)-component theorem.
First, we describe Welter's game as a game with Young diagrams.
We then state the \(p\)'-component theorem.
Finally, we generalize this theorem by using Proposition \ref{org14ff579}.

In this section, we assume that \(\ccp\) is a prime.

\subsection{Welter's game and Young diagrams}
\label{sec:org97d9085}
\label{org1dc745b}

For \(n \in \NN\), a \emph{partition} \(\lambda\) of \(n\) is a tuple \((\lambda^1, \ldots, \lambda^m)\) of positive integers such that \(\sum \lambda^i = n\) and \(\lambda^1 \ge \cdots \ge \lambda^m\).
For example, \((4, 4, 1)\) is a partition of 9 and \(()\) is a partition of 0. 
If \(\lambda\) is a partition \((\lambda^1, \ldots, \lambda^m)\), then
\[
 \set{(i, j) \in \NN^2 : 1 \le \cci \le \ccm,\ 1 \le \ccj \le \lambda^\cci}
\]
is called the \emph{Young diagram} or the \emph{Ferrers diagram} corresponding to \(\lambda\).
We will identify a partition with its Young diagram.
A Young diagram \(\ccY\) can be visualized by using \(\size{\ccY}\) cells. For example,
Figure \ref{fig:orgc88e075} shows the Young diagram \((5, 4, 3)\).

\begin{figure}[htbp]
\centering
\includegraphics[scale=0.6]{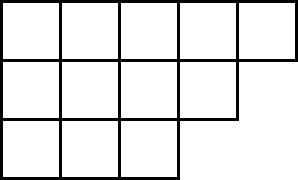}
\caption{\label{fig:orgc88e075}The Young diagram (5,4,3).}
\end{figure}

For a position \(A\) in \(\Welter[\ccm]\), let
\[
 \ccY(\ccA) = (\cca^{\sigma(1)} - \ccm + 1, \cca^{\sigma(2)} - \ccm + 2, \ldots, \cca^{\sigma(\ccm)}),
\]
where \(\sigma\) is a permutation with \(\cca^{\sigma(1)} > \cca^{\sigma(2)} > \cdots > \cca^{\sigma(\ccm)}\).
We consider \(\ccY(\ccA)\) as a Young diagram by ignoring zeros in \(\ccY(\ccA)\).
For example, if \(\ccA = (3, 7, 5)\) and \(\ccA' = (5, 9, 7, 1, 0)\), then
\[
 \ccY(\ccA) = (7 - 2, 5 - 1, 3) = (5, 4, 3) \tand \ccY(\ccA') = (9 - 4, 7 - 3, 5 - 2, 1 - 1, 0) = (5, 4, 3, 0, 0),
\]
so \(\ccY(\ccA') = \ccY(\ccA) = \Yvcentermath1 \stiny\yng(5,4,3)\).
Note that the number of cells in \(\ccY(\ccA)\) is equal to \(\lg_{\Welter[\ccm]}(\ccA)\) since
\[
 \lg_{\Welter[\ccm]}(\ccA) = \sum_{\cci} \cca^i - (1 + \cdots + \ccm - 1) =  \size{\ccY(\ccA)}.
\]
In the above example,
\[
 \lg_{\Welter[3]}(\ccA) = 3 + 7 + 5 - (1 + 2) = (7 - 2) + (5 - 1) + 3 = 12 = \size{\ccY(\ccA)}.
\]

It is known that moving in Welter's game corresponds to removing a hook.
Here, for a cell \((\cci,\ccj)\) in a Young diagram \(\ccY\), the \emph{\((\cci, \ccj)\)-hook} \(\ccH_{ij}(\ccY)\) is defined by
\[
 \ccH_{\cci, \ccj}(\ccY) = \set{(\cci', \ccj') \in \ccY : (\cci' \ge \cci \tand \ccj' = \ccj) \tor (\cci' = \cci \tand \ccj' \ge \ccj)}.
\]
In other words, \(\ccH_{\cci, \ccj}(\ccY)\) consists of the cells to the right of \((\cci, \ccj)\), the cells below \((\cci, \ccj)\), and \((\cci, \ccj)\) itself.
For example, Figure  \ref{fig:org3a7d2ee} shows the \((1,2)\)-hook of \((5, 4, 3)\).

\begin{figure}[htbp]
\centering
\includegraphics[scale=0.6]{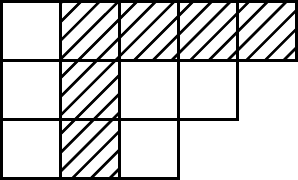}
\caption{\label{fig:org3a7d2ee}The (1,2)-hook of (5,4,3).}
\end{figure}

We now describe removing a hook. We first remove cells in \(\ccH_{ij}(\ccY)\) from \(\ccY\).
If we get two diagrams, then we next push them together. The obtained Young diagram is denoted by
\(\ccY \setminus \ccH_{ij}(\ccY)\) and is said to be \emph{obtained from \(\ccY\) by removing the \((\cci, \ccj)\)-hook}.
For example, if \(\ccY = (5, 4,3)\), 
then \(\ccY \setminus \ccH_{1,2}(\ccY) = (3, 2, 1)\) and \(\ccY \setminus \ccH_{3,2}(\ccY) = (5, 4, 1)\)
(see Figure \ref{fig:org951dc8f}).

\begin{figure}[htbp]
\centering
\includegraphics[scale=0.6]{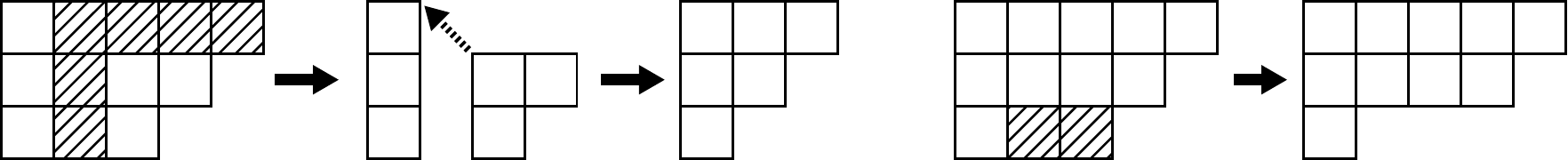}
\caption{\label{fig:org951dc8f}Removing the \((1,2)\)-hook and the \((3, 2)\)-hook.}
\end{figure}

Moving in Welter's game corresponds to removing a hook as follows.
Let \(\ccA\) be a position in \(\Welter[\ccm]\) and \(\ccB\) its option.
Then we can write
\[
 \ccB = (\cca^1, \ldots, \cca^{\ccs - 1}, \cca^s - \cch, \cca^{\ccs + 1}, \ldots, \cca^\ccm).
\]
Let \(\cca^\ccs\) be the \(i\)th largest element in the components of \(\ccA\).
Note that \(\cca^\ccs - \cch \in \NN \setminus \set{\cca^1, \ldots, \cca^\ccm}\).
Suppose that \(\cca^\ccs - \cch\) is the \(j\)th smallest element in \(\NN \setminus \set{\cca^1, \ldots, \cca^\ccm}\).
Then the following holds (see, for example, \cite{olsson-Combinatorics-1993}):
\begin{equation}
\label{orgfa24ddf}
 \ccY(\ccB) = \ccY(\ccA) \setminus \ccH_{\cci, \ccj}(\ccY(\ccA)).
\end{equation}
For example, consider moving from \((7, 5, 3)\) to \((1, 5, 3)\) in \(\Welter[3]\).
Since 7 is the largest element in \(\set{7, 5, 3}\) and 1 is the second smallest element in \(\NN \setminus \set{7, 5, 3}\),
it follows from (\ref{orgfa24ddf}) that
\[
 \ccY(1, 5, 3) = \ccY(7, 5,3) \setminus \ccH_{1,2}(\ccY(7,5,3)).
\]
Indeed, \(\ccY(1, 5, 3)\) is the Young diagram \((3, 2, 1)\) and
\[
 \ccY(7, 5, 3) \setminus \ccH_{1,2}(\ccY(7,5,3)) = (5, 4, 3) \setminus \ccH_{1,2}(5,4,3) = (3, 2, 1)
\]
as we have seen in Figure \ref{fig:org951dc8f}.
In this way, moving in Welter's game corresponds to removing a hook.
Note that it is obvious that \(\size{\ccY(\ccA)} = \lg_{\Welter[\ccm]}(\ccA)\).

The number of cells in the \((\cci, \ccj)\)-hook is called the \emph{hook-length} of \((\cci, \ccj)\).
Figure \ref{fig:orge7d6ff3} shows the hook-lengths of the Young diagram \((5,4,3)\).
Let \(\sH(\ccY)\) be the multiset of hook lengths of \(\ccY\).
For example, \(\sH(5,4,3) = \set{1,1,1,2,3,3,3,4,5,5,6,7}\).

\begin{figure}[htbp]
\centering
\includegraphics[scale=0.6]{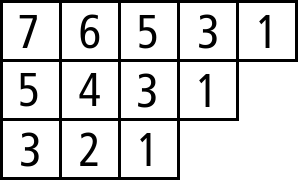}
\caption{\label{fig:orge7d6ff3}The hook lengths of (5, 4, 3).}
\end{figure}

 \begin{theorem}[\hspace{0.01em}\cite{sato-game-1968, sato-mathematical-1970, sato-maya-1970}]
 \comment{Thm. [\hspace{0.01em}\cite{sato-game-1968, sato-mathematical-1970, sato-maya-1970}]}
\label{sec:org747c946}
\label{orgd359077}
If \(\ccA\) is a position in Welter's game \(\Welter[\ccm]\),
then
\[
 \sg_{\Welter[\ccm]}(\ccA) = \bigopluspm[\cch \in \sH(\ccY)][][2][-1.1em][0.3em] \pnorm<2>(\cch),
\]
where \(\ccY = \ccY(\ccA)\) and \(\pnorm<2>(\cch) = \sum_{\ccL = 0}^{\ord_2(\cch)} 2^\ccL = \pexp<2>{\underbrace{1, \ldots, 1}_{\ord_2(\cch) + 1}}\).
 
\end{theorem}

 \begin{example}
 \comment{Exm.}
\label{sec:orgf800d1a}
\label{orgb52c2ba}
If \(\ccA = (7, 5, 3)\), then
\begin{align*}
 \sg_{\Welter[3]}(\ccA) &= \bigopluspm[\cch \in \sH(\ccY(\ccA))][][2][-1.6em][0.5em] \pnorm<2>(\cch) \\
 &= \pnorm<2>(1) \oplus_2 \pnorm<2>(1) \oplus_2 \pnorm<2>(1) \oplus_2 \pnorm<2>(3) \oplus_2 \pnorm<2>(3) \oplus_2 \pnorm<2>(3) \\
 &\ \ \ \ \oplus_2\pnorm<2>(5) \oplus_2 \pnorm<2>(5) \oplus_2 \pnorm<2>(7) \\
 &\ \ \ \  \oplus_2 \pnorm<2>(2) \oplus_2 \pnorm<2>(6) \oplus_2 \pnorm<2>(4) \\
 &= 1 \oplus_2 1 \oplus_2 1 \oplus_2 1 \oplus_2 1 \oplus_2 1 \oplus_2 1 \oplus_21 \oplus_2 1 \oplus_2 3 \oplus_2 3 \oplus_2 7 \\
 &= 6.
\end{align*}
 
\end{example}

\comment{connect}
\label{sec:org7c750ff}
From Theorem \ref{org28c6b37}, we obtain the following analogue of Theorem \ref{orgd359077}.\footnote{Theorem \ref{orgf794c12} holds for an integer \(\ccp\) greater than 1.}

 \begin{theorem}[\hspace{0.01em}\cite{irie-pSaturations-2018a}]
 \comment{Thm. [\hspace{0.01em}\cite{irie-pSaturations-2018a}]}
\label{sec:org9ce53b0}
\label{orgf794c12}
Let \(\Gamma\) be a \(p\)-saturation of \(\Welter[\ccm]\). 
If \(\ccA\) is a position in \(\Gamma\), then
\begin{equation}
\label{orga301ffb}
 \sg_{\Gamma}(\ccA) = \bigopluspm[\cch \in \sH(\ccY)][][\ccp][-1.1em][0.3em] \pnorm<\ccp>(\cch),
\end{equation}
where \(\ccY = \ccY(\ccA)\) and \(\pnorm<\ccp>(\cch) = \sum_{\ccL = 0}^{\ord_\ccp(\cch)} \ccp^\ccL  = \pexp<\ccp>{\underbrace{1, \ldots, 1}_{\ord_\ccp(\cch) + 1}}\).
 
\end{theorem}

 \begin{example}
 \comment{Exm.}
\label{sec:orgefc9b0b}
\label{org31ec141}
Let \(\Gamma\) be a \(5\)-saturation of \(\Welter[3]\) and
\(\ccA\) be the position (7, 5, 3) in \(\Gamma\).
By Theorem \ref{org28c6b37}, we see that
\begin{align*}
 \sg_{\Gamma}(\ccA) &= \bigopluspm[\cch \in \sH(\ccY(\ccA))][][5][-1.6em][0.6em] \pnorm<5>(\cch) \\
 &= \pnorm<5>(1) \oplus_5 \pnorm<5>(1) \oplus_5 \pnorm<5>(1) \oplus_5 \pnorm<5>(2) \oplus_5 \pnorm<5>(3)  \\
 &\ \ \ \ \oplus_5 \pnorm<5>(3) \oplus_5 \pnorm<5>(3) \oplus_5 \pnorm<5>(4) \oplus_5 \pnorm<5>(6) \oplus_5 \pnorm<5>(7)  \\
 &\ \ \ \ \oplus_5 \pnorm<5>(5) \oplus_5 \pnorm<5>(5)  \\
 &= \pnorm<5>(5) \oplus_5 \pnorm<5>(5)  \\
 &= 6 \oplus_5 6 \\
 &= 12.
\end{align*}
 
\end{example}

\subsection{The \(p'\)-component theorem}
\label{sec:org638d4da}
\label{org85bc411}
We describe the \(p'\)-component theorem in terms of Young tableaux. 

Let \(\ccY\) be a Young diagram. A \emph{Young tableau} \(\ccT\) of shape \(\ccY\) is
a bijection from \(\ccY\) to \(\set{1, 2, \ldots, \size{\ccY}}\).
A Young tableau \(\ccT\) of shape \(\ccY\) is called a \emph{standard tableau} if \(\ccT(\cci, \ccj) \le \ccT(\cci', \ccj')\) for 
\((\cci, \ccj), (\cci', \ccj') \in \ccY\) with \(\cci \le \cci'\) and \(\ccj \le \ccj'\).
We can visualize a Young tableau by writing numbers in cells.
Figure \ref{fig:org17993c7} shows an example of a standard tableau.

\begin{figure}[htbp]
\centering
\includegraphics[scale=0.6]{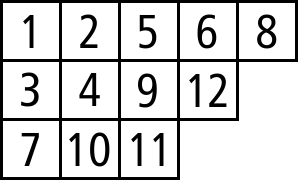}
\caption{\label{fig:org17993c7}A standard tableau of shape \((5, 4, 3)\).}
\end{figure}

Let \(\ccf^\ccY\) be the number of standard tableaux of shape \(\ccY\).
We can calculate \(\ccf^\ccY\) by using the hook lengths of \(\ccY\).

 \begin{theorem}[Hook Formula \cite{frame-hook-1954}]
 \comment{Thm. [Hook Formula \cite{frame-hook-1954}]}
\label{sec:org3c33193}
\label{org94584db}
If \(\ccY\) is a Young diagram with \(\ccn\) cells, then
\[
 f^\ccY = \frac{n!}{\prod_{\cch \in \sH(\ccY)} \cch}.
\]
 
\end{theorem}

 \begin{example}
 \comment{Exm.}
\label{sec:org9382461}
\label{org88a155f}
Let \(\ccY = (2, 1)\). By Theorem \ref{org94584db},
\[
 f^\ccY = \frac{3!}{1^2 \cdot 3} = 2.
\]
Indeed, there are exactly two standard tableaux of shape \(\ccY\) (see Figure \ref{fig:orgb68f321}).
\begin{figure}[htbp]
\centering
\includegraphics[scale=0.6]{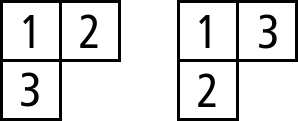}
\caption{\label{fig:orgb68f321}Two standard tableaux of shape \((2, 1)\).}
\end{figure}
 
\end{example}

\begin{remark}
 \comment{Rem.}
\label{sec:org5ed25a5}
\label{org694c08b}
The above formula is important in combinatorial representation theory.
Let \(\Sym(\ccn)\) denote the symmetric group on \(\set{1,2,\ldots,\ccn}\).
A \emph{representation} \(\cRep\) of \(\Sym(\ccn)\) is a group homomorphism from \(\Sym(\ccn)\)
to the general linear group \(GL(d, \CC)\), the group of invertible \(d \times d\) matrices with entries in \(\CC\).
The number \(\ccd\) is called the \emph{degree} of \(\rho\).
It is known that the Young diagrams with \(n\) cells are in one-to-one correspondence with the irreducible\footnote{Roughly speaking, an irreducible representation is kind of an atom among representations because every representation can be decomposed into a direct sum of irreducible representations.} representations of \(\Sym(n)\).
For a Young diagram \(Y\) with \(n\) cells, let \(\cRep^Y\) denote the irreducible representation of \(\Sym(n)\) corresponding to \(Y\).
Then the degree of \(\cRep^\ccY\) is equal to \(\ccf^\ccY\).
In addition to degrees, a lot of results about representations of symmetric groups can be approached in a purely combinatorial way.
See, for example, \cite{sagan-Symmetric-2001} for details.
 
\end{remark}

\comment{connect}
\label{sec:org4ba8466}
Using \(\ccf^\ccY\), we restate Proposition \ref{org8a24e5c}.
Macdonald \cite{macdonald-Degrees-1971} characterized \(\ccY\) such that \(\ccf^\ccY\) is prime to \(\ccp\).\footnote{Macdonald's characterization holds only when \(\ccp\) is prime, so we have to assume that \(\ccp\) is prime.}
From this characterization, we can show that
\(\ccf^\ccY\) is prime to \(\ccp\) if and only if \(\ppsi<\ccp>(\ccY) = \size{\ccY}\),
where \(\ppsi<\ccp>(\ccY)\) is the right-hand side of (\ref{orga301ffb}).
Recall that 
\[
 \ppsi<\ccp>(\ccY(\ccA)) = \sg_{\tilde{\Gamma}}(\ccA) \quad \tand \quad \size{\ccY(\ccA)} = \lg_{\tilde{\Gamma}}(\ccA)
\]
for a position \(\ccA\) in a \(p\)-saturation \(\tilde{\Gamma}\) of \(\Welter[\ccm]\).
Hence \(\ccA\) is full if and only if \(\ccf^{\ccY(\ccA)}\) is prime to \(\ccp\).
Therefore the following result follows from Proposition \ref{org8a24e5c}.

 \begin{theorem}[\hspace{0.01em}\cite{irie-pSaturations-2018a}]
 \comment{Thm. [\hspace{0.01em}\cite{irie-pSaturations-2018a}]}
\label{sec:orgd063ab3}
\label{org2374258}
Every Young diagram \(\ccY\) includes a Young diagram \(\ccZ\) with \(\ppsi<\ccp>(\ccY)\) cells such that \(\ccf^\ccZ\) is prime to \(\ccp\).
 
\end{theorem}

\begin{proof}
 \comment{Proof.}
\label{sec:orge8a3944}
Let \(\ccY = (\lambda^1, \lambda^2, \ldots, \lambda^\ccm)\) and \(\ccA\) be the position \((\lambda^1 + \ccm - 1, \lambda^2 + \ccm - 2, \ldots, \lambda^\ccm)\)
in a \(p\)-saturation \(\tilde{\Gamma}\) of \(\Welter[\ccm]\).
Then \(\ccY(\ccA) = \ccY\). By Proposition \ref{org8a24e5c},
the position \(\ccA\) has a full descendant \(\ccB\) with the same Sprague-Grundy value as \(\ccA\), that is, \(\lg_{\tilde{\Gamma}}(\ccB) = \ppsi<\ccp>(\ccB) = \ppsi<\ccp>(\ccA)\).
Since moving in Welter's game corresponds to removing a hook,
we see that \(\ccY(\ccB) \subseteq \ccY(\ccA)\). Moreover,
\[
 \size{\ccY(\ccB)} = \lg_{\tilde{\Gamma}}(\ccB) = \ppsi<\ccp>(\ccB) = \ppsi<\ccp>(\ccA) = \ppsi<\ccp>(\ccY(\ccA)).
\]
Therefore \(\ccY(\ccB)\) is the desired Young diagram.
\end{proof}

 \begin{example}
 \comment{Exm.}
\label{sec:orgc174c7d}
\label{org62864b3}
Let \(\ccp = 2\), and let \(\ccY\) be the Young diagram \(\Yvcentermath1 \stiny\yng(4,3,2)\).
By Theorem \ref{org94584db},
\[
 \ccf^\ccY = \frac{9!}{1^3 \cdot 2 \cdot 3^2 \cdot 4 \cdot 5 \cdot 6} = 168.
\]
Note that \(\ccf^\ccY\) is even.
Now \(\ccY\) corresponds to the position \((6, 4, 2)\) in \(\Welter[3]\).
Since \(\ppsi<2>(\ccY) = \sg_{\Welter[3]}(6, 4, 2) = 7\),
Theorem \ref{org2374258} says that
\(\ccY\) includes a Young diagram \(\ccZ\) with 7 cells such that \(\ccf^\ccZ\) is odd.
Indeed, as we have seen in Example \ref{org9f6734a}, 
the position \((5, 3, 2)\) is a full descendant of
\((6, 4, 2)\) and these two positions have the same Sprague-Grundy value.
Let \(\ccZ = \Yvcentermath1\ccY(5, 3, 2) = \stiny\yng(3,2,2)\). Then \(\ccZ \subset \ccY\) and
\[
 \size{\ccZ} = 3 + 2 + 2 = 7 = \ppsi<2>(\ccY).
\]
Moreover,
\[
 \ccf^\ccZ = \frac{7!}{1^2 \cdot 2^2 \cdot 3 \cdot 4 \cdot 5} = 21.
\]
In particular, \(\ccf^\ccZ\) is odd.
 
\end{example}

\begin{remark}
 \comment{Rem.}
\label{sec:org8ead47b}
\label{org5df1f74}
Theorem \ref{org2374258} has the following algebraic interpretation.
Let \(\ccY\) be a Young diagram with \(\ccn\) cells.
Recall that the corresponding irreducible representation \(\cRep^\ccY\)
is a map from \(\Sym(\ccn)\) to \(GL(d, \CC)\).
Since \(\Sym(\ccn - 1) \subseteq \Sym(\ccn)\),
we can obtain a representation of \(\Sym(\ccn - 1)\) by restricting \(\cRep^\ccY\) to \(\Sym(\ccn - 1)\).
The obtained representation \(\cRep^\ccY\big|_{\Sym(\ccn - 1)}\) may not be irreducible and can be decomposed as follows:
\begin{equation}
\label{org30725d4}
 \cRep^\ccY\big|_{\Sym(\ccn - 1)} = \bigoplus_{\ccY^-} \cRep^{\ccY^-},
\end{equation}
where the direct sum runs over all Young diagrams \(Y^-\) obtained from \(Y\) by removing a hook of length 1.
For example, 
\[
 \Yvcentermath1\cRep^{\,\sstiny \yng(2,1)}\big|_{\Sym(2)} = \cRep^{\,\sstiny \yng(1,1)} \oplus \cRep^{\,\sstiny \yng(2)}
 \quad \tand \quad \cRep^{\,\sstiny \yng(4,3,2)}\Big|_{\Sym(8)} = \cRep^{\,\sstiny \yng(3,3,2)}  \oplus \cRep^{\,\sstiny \yng(4,2,2)} \oplus \cRep^{\,\sstiny \yng(4,3,1)}.
\]
From (\ref{org30725d4}) and Theorem \ref{org2374258}, we see that
the restriction \(\cRep[\ccY]\) to \(\Sym(\ppsi<\ccp>(\ccY))\) has a component with degree prime to \(\ccp\).
Thus we will call Theorem \ref{org2374258} the \emph{\(p'\)-component theorem}.
For example, if \(\ccY = (4, 3, 2)\), then
\begin{align*}
 \Yvcentermath1 \cRep^{\ccY}\big|_{\Sym(\ppsi<2>(\ccY))} &= \Yvcentermath1\cRep^{\,\sstiny \yng(4,3,2)}\Big|_{\Sym(7)} 
 \Yvcentermath1= \Bigl(\cRep^{\,\sstiny \yng(4,3,2)}\Big|_{\Sym(8)}\Bigr)\Big|_{\Sym(7)} \\
 &= \Yvcentermath1\Bigl(\cRep^{\,\sstiny \yng(3,3,2)}  \oplus \cRep^{\,\sstiny \yng(4,2,2)} \oplus \cRep^{\,\sstiny \yng(4,3,1)}\Bigr)\Big|_{\Sym(7)} \\
 &= \Yvcentermath1\Bigl(\cRep^{\,\sstiny \yng(3,2,2)} \oplus \cRep^{\,\sstiny \yng(3,3,1)} \Bigr) \oplus \Bigl(\cRep^{\,\sstiny \yng(3,2,2)} \oplus \cRep^{\,\sstiny \yng(4,2,1)} \Bigr) \oplus \Bigl(\cRep^{\,\sstiny \yng(3,3,1)} \oplus \cRep^{\,\sstiny \yng(4,2,1)} \oplus \cRep^{\,\sstiny \yng(4,3)} \Bigr).  
\end{align*}
Since
\[
 \Yvcentermath1\deg \cRep^{\,\sstiny \yng(3,2,2)} = \ccf^{\,\sstiny \yng(3,2,2)} = 21,
\]
the representation \(\Yvcentermath1\cRep^{\,\sstiny \yng(3,2,2)}\) is a component of \(\cRep^{\ccY}|_{\Sym(\ppsi<2>(\ccY))}\) with odd degree.
 
\end{remark}

\comment{connect}
\label{sec:org86e5cae}
As we have mentioned, by using the \(p'\)-component theorem, we can show that
the Sprague-Grundy function of a \(p\)-saturation of Welter's game is equal to \(\ppsi<\ccp>\).
We generalize the \(p'\)-component theorem in the next section.

\subsection{A generalization of the \(p'\)-component theorem}
\label{sec:orgaf5d765}
Let \(\bbY\) be a \(\cck\)-tuple \((\ccY^1, \ldots, \ccY^\cck)\) of Young diagrams.
A \emph{Young tableau} \(\bbT\) of shape \(\bbY\) is a bijection from the disjoint union \(\bigsqcup_{\ccl} \ccY^l\)  to \(\set{1, 2, \ldots \size{\bbY}}\),
where \(\size{\bbY} = \sum_\ccl \size{\ccY^\ccl}\).
A Young tableau \(\bbT\) of shape \(\bbY\) is called a \emph{standard tableau} if \(\bbT(\cci, \ccj) \le \bbT(\cci', \ccj')\) for 
\((\cci, \ccj), (\cci', \ccj') \in \ccY^\ccl\) and \(\ccl \in \set{1,\ldots,\cck}\) with \(\cci \le \cci'\) and \(\ccj \le \ccj'\).
Figure \ref{fig:org1394d7d} shows an example of a standard tableau.

\begin{figure}[htbp]
\centering
\includegraphics[scale=0.6]{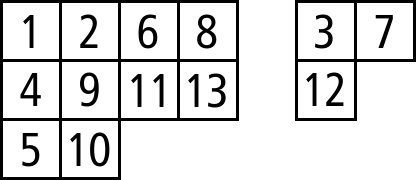}
\caption{\label{fig:org1394d7d}A standard tableau of shape \(((4,4,2), (2,1))\).}
\end{figure}

Let \(f^\bbY\) denote the number of standard tableaux of shape \(\bbY\).
It follows from Theorem \ref{org94584db} that
\[
 f^\bbY = \frac{\bbn!}{n^1! \cdots n^\cck !} \frac{n^1 !}{ \prod_{\cch \in \sH(\ccY^1)} \cch} \cdots \frac{n^\cck !}{ \prod_{\cch \in \sH(\ccY^\cck)} \cch} = \frac{\bbn!}{\prod_{\cch \in \sH(\bbY)} \cch},
\]
where \(\bbn = \size{\bbY}\), \(\ccn^\ccl = \size{\ccY^\ccl}\), and \(\sH(\bbY)\) is the multiset addition \(\sH(\ccY^1) + \ldots + \sH(\ccY^\cck)\).
For example, \(\set{1, 2} + \set{1, 3} = \set{1, 1, 2, 3}\).
Moreover, by using Olsson's generalization \cite{olsson-mckay-1976} of Macdonald's characterization, we can show that
\(f^\bbY\) is prime to \(\ccp\) if and only if \(\ppsi<\ccp>(\bbY) = \size{\bbY}\), where 
\begin{align*}
 \ppsi<\ccp>(\bbY) &= \ppsi<\ccp>(\ccY^1) \oplus_\ccp \cdots \oplus_\ccp \ppsi<\ccp>(\ccY^\cck) \\
 &= \bigopluspm[\cch \in \sH(\bbY)][][\ccp][-1.1em][0.3em] \pnorm<\ccp>(\cch).
\end{align*}

Let \(\bbY\) and \(\bbZ\) be \(\cck\)-tuples \((\ccY^1, \ldots, \ccY^\cck)\) and \((\ccZ^1, \ldots, \ccZ^\cck)\) of Young diagrams, respectively.
We write \(\bbZ \subseteq \bbY\) if \(\ccZ^\ccl \subseteq \ccY^\ccl\) for \(\ccl \in \set{1, \ldots, \cck}\).
The following theorem follows from Proposition \ref{org14ff579}.

 \begin{theorem}
 \comment{Thm.}
\label{sec:orgf8a75f4}
\label{org4fefb85}
Every \(\cck\)-tuple \(\bbY\) of Young diagrams includes a \(\cck\)-tuple \(\bbZ\) of Young diagrams 
with \(\ppsi<\ccp>(\bbY)\) cells in total such that \(\ccf^\bbZ\) is prime to \(\ccp\).
 
\end{theorem}

\begin{proof}
 \comment{Proof.}
\label{sec:org71fb998}
Let \(\bbY = (\ccY^1, \ldots, \ccY^\cck)\), \(\ccY^\ccl = (\lambda^{\ccl, 1}, \ldots, \lambda^{\ccl, \ccm^\ccl})\), \(\ccA^\ccl = (\lambda^{l, 1} + \ccm^\ccl - 1,\) \(\lambda^{l, 2} + \ccm^\ccl - 2,\) \(\ldots,\) \(\lambda^{l, \ccm^\ccl})\),
and \(\bbA = (\ccA^1, \ldots, \ccA^\cck)\). 
Then \(\bbY = (\ccY(\ccA^1), \ldots, \ccY(\ccA^\cck))\).
We consider \(\bbA\) as a position in \(\tilde{\Gamma}\),
where \(\tilde{\Gamma}\) is a \(p\)-saturation of \(\Welter[\ccm^1] + \cdots + \Welter[\ccm^\cck]\).
By Proposition \ref{org14ff579},
the position \(\bbA\) has a full descendant \(\bbB\) with the same Sprague-Grundy value as \(\bbA\), that is, \(\lg_{\tilde{\Gamma}}(\bbB) = \ppsi<\ccp>(\bbB) = \ppsi<\ccp>(\bbA)\).
Let \(\bbZ = (\ccY(\ccB^1), \ldots, \ccY(\ccB^\cck))\).
We see that \(\bbZ \subseteq \bbY\). Moreover,
\[
 \size{\bbZ} = \lg_{\tilde{\Gamma}}(\bbB) = \ppsi<\ccp>(\bbB) = \ppsi<\ccp>(\bbA) = \ppsi<\ccp>(\bbY).
\]
Therefore \(\bbZ\) is the desired \(\cck\)-tuple of Young diagrams.
\end{proof}

 \begin{example}
 \comment{Exm.}
\label{sec:orga60059a}
\label{org4aa6273}
Let \(\ccp = 2\) and \(\Yvcentermath1 \bbY = ((4,4,2), (2,1)) = \Bigl(\,\stiny \yng(4,4,2),\ \ \raisebox{0.1em}{\yng(2,1)}\Bigr)\).
Then 
\[
 \sH(\bbY) = \set{1,1,1,1,2,2,2,3,3,4,5,5,6} 
\]
(see Figure \ref{fig:orge64af80}). Hence
\[
 \Yvcentermath1 \ccf^{\bigl(\,\sstiny \yng(4,4,2),\ \ \raisebox{0.2em}{\yng(2,1)}\bigr)} = \frac{13!}{1^4 \cdot 2^3 \cdot 3^2 \cdot 4 \cdot 5^2 \cdot 6} = 144144.
\]
Moreover,
\[
 \Yvcentermath1\ppsi<2>\Bigl(\,{\stiny\yng(4,4,2),\ \ \raisebox{0.1em}{\yng(2,1)}}\Bigr) = \ppsi<2>\Bigl({\stiny\yng(4,4,2)}\Bigr) \oplus_2 \ppsi<2>\Bigl({\stiny\yng(2,1)}\Bigr) = 6 \oplus_2 1 = 7.
\]
By Theorem \ref{org4fefb85}, \(\bbY\) includes a pair \(\bbZ\) of Young diagrams with 7 cells in total such that \(\ccf^\bbZ\) is odd. 
Indeed, let \(\bbZ = ((2,2,1), (2))\). We see that
\[
 \Yvcentermath1 \ccf^{\,\sstiny \bigl(\, \yng(2,2,1),\ \ \raisebox{1.4em}{\yng(2)}\bigr)} = \frac{7!}{1^3 \cdot 2^2 \cdot 3 \cdot 4} = 105.
\]
Thus \(\bbZ\) is the desired pair of Young diagrams.
\begin{figure}[htbp]
\centering
\includegraphics[scale=0.6]{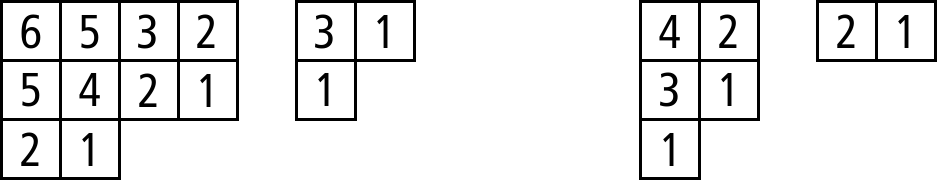}
\caption{\label{fig:orge64af80}The hook lengths of \(((4,4,2), (2,1))\) and \(((2,2,1), (2))\).}
\end{figure}
 
\end{example}

\begin{remark}
 \comment{Rem.}
\label{sec:orgada34ad}
\label{org09152e2}
We give an algebraic interpretation of Theorem \ref{org2374258}.
It is known that \(\cck\)-tuples of Young diagrams with \(n\) cells in total 
are in one-to-one correspondence with the irreducible representations of the generalized symmetric group \((\mathbb{Z} / \cck \mathbb{Z}) \wr \Sym(\ccn)\) (see, for example, \cite{osima-representations-1954, stembridge-eigenvalues-1989} for details).
For a \(\cck\)-tuple \(\bbY\) of Young diagrams with \(\ccn\) cells in total, let \(\cRep^\bbY\) denote the corresponding irreducible representation of \((\mathbb{Z} / \cck \mathbb{Z}) \wr \Sym(n)\).
Then the degree of \(\cRep^\bbY\) is equal to \(\ccf^\bbY\).
Moreover, the following analogue of (\ref{org30725d4}) holds:
\begin{equation}
\label{org80d910d}
 \cRep^\bbY\big|_{(\mathbb{Z} / \cck \mathbb{Z}) \wr \Sym(\ccn - 1)} = \bigoplus_{\bbY^-} \cRep^{\bbY^-},
\end{equation}
where the direct sum runs over all \(k\)-tuples \(\bbY^-\) of Young diagrams obtained from \(\bbY\) by removing a hook of length 1.
For example,
\[
\Yvcentermath1 \cRep^{\bigl(\,\sstiny \yng(4,4,2),\ \ \raisebox{0.2em}{\yng(2,1)}\bigr)} \big|_{(\ZZ/2\ZZ) \wr \Sym(12)} = \cRep^{\bigl(\,\sstiny \yng(4,3,2),\ \ \raisebox{0.2em}{\yng(2,1)}\bigr)} \oplus \cRep^{\bigl(\,\sstiny \yng(4,4,1),\ \ \raisebox{0.2em}{\yng(2,1)}\bigr)} \oplus \cRep^{\bigl(\,\sstiny \yng(4,4,2),\ \ \raisebox{0.2em}{\yng(1,1)}\bigr)} \oplus \cRep^{\bigl(\,\sstiny \yng(4,4,2),\ \ \raisebox{1.4em}{\yng(2)}\bigr)}.
\]
Let \(\bbY\) and \(\bbZ\) be two \(\cck\)-tuples of Young diagrams such that \(\bbZ \subseteq \bbY\).
By (\ref{org80d910d}), the representation \(\cRep[\bbZ]\) is a component of
the restriction of \(\cRep[\bbY]\) to \((\mathbb{Z} / \cck \mathbb{Z}) \wr \Sym(\size{\bbZ})\).
Therefore Theorem \ref{org4fefb85} says that
the restriction \(\cRep[\bbY]\) to \((\mathbb{Z} / \cck \mathbb{Z}) \wr \Sym(\ppsi<\ccp>(\bbY))\) has a component with degree prime to \(\ccp\).
 
\end{remark}

\comment{References}
\label{sec:orgbc4db5b}


\begin{thebibliography}{10}\footnotesize 

\bibitem{berlekamp-Winning-2001}
E.~R. Berlekamp, J.~H. Conway, and R.~K. Guy,
\newblock {\em Winning {{Ways}} for {{Your Mathematical Plays}}},
\newblock {A.K. Peters}, {Natick, Mass.}, second edition, 2001.

\bibitem{blass-how-1998}
U.~Blass, A.~S. Fraenkel, and R.~Guelman,
\newblock How {{far can Nim}} in {{disguise be stretched}}?,
\newblock {\em J. Combin. Theory Ser. A} {\bf 84:2} (1998),
  145--156.

\bibitem{conway-numbers-2001}
J.~H. Conway,
\newblock {\em On Numbers and Games},
\newblock {A.K. Peters}, {Natick, Mass.}, second edition, 2001.

\bibitem{frame-hook-1954}
J.~S. Frame, G.~de B. Robinson, and R.~M. Thrall,
\newblock The hook graphs of the symmetric group,
\newblock {\em Canad. J. Math.} {\bf 6} (1954), 316--324.

\bibitem{grundy-mathematics-1939}
P.~M. Grundy,
\newblock Mathematics and games,
\newblock {\em Eureka} {\bf 2} (1939), 6--8.

\bibitem{irie-pSaturations-2018a}
Y.~Irie,
\newblock {{{\emph{p}}-Saturations of Welter's game and the irreducible
  representations of symmetric groups}},
\newblock {\em J. Algebraic Combin.} {\bf 48} (2018), 247--287.

\bibitem{li-nperson-1978}
S.~Y.~R. Li,
\newblock N-person {{Nim}} and N-person {{Moore}}'s {{games}},
\newblock {\em Internat. J. Game Theory} {\bf 7:1} (1978), 31--36.

\bibitem{macdonald-Degrees-1971}
I.~G. Macdonald,
\newblock On the {{degrees}} of the {{irreducible representations}} of
  {{symmetric groups}},
\newblock {\em Bull. Lond. Math. Soc.} {\bf 3:2} (1971),
  189--192.

\bibitem{moore-generalization-1910}
E.~H. Moore,
\newblock A {{generalization}} of the {{game called Nim}},
\newblock {\em Ann. of Math.} {\bf 11:3} (1910), 93--94.

\bibitem{navarro-character-2018}
G.~Navarro,
\newblock {\em Character {{Theory}} and the {{McKay Conjecture}}},
\newblock Cambridge {{Studies}} in {{Advanced Mathematics}}, {Cambridge
  University Press}, {Cambridge}, 2018.

\bibitem{olsson-mckay-1976}
J.~B. Olsson,
\newblock {{McKay}} numbers and heights of characters,
\newblock {\em Math. Scand.} {\bf 38} (1976), 25--42.

\bibitem{olsson-Combinatorics-1993}
J.~B. Olsson,
\newblock {\em Combinatorics and {{Representations}} of {{Finite Groups}}},
\newblock {Vorlesungen aus dem FB Mathematik der Univ. Essen, Heft 20}, 1993.

\bibitem{osima-representations-1954}
M.~Osima,
\newblock On the representations of the generalized symmetric group,
\newblock {\em Math. J. Okayama Univ.} {\bf 4:1} (1954), 39--56.

\bibitem{sagan-Symmetric-2001}
B.~Sagan,
\newblock {\em The {{Symmetric Group}}: {{Representations}}, {{Combinatorial
  Algorithms}}, and {{Symmetric Functions}}},
\newblock Graduate {{Texts}} in {{Mathematics}}, {Springer-Verlag}, {New York},
  second edition, 2001.

\bibitem{sato-game-1968}
M.~Sato,
\newblock On a game (notes by {{Kenji Ueno}})(in {{Japanese}}),
\newblock In {\em Proc. of the 12th Symposium of the {{Algebra Section}}
  of the {{Math. Soc. Japan}}} (1968), 123--136.

\bibitem{sato-mathematical-1970}
M.~Sato,
\newblock Mathematical theory of {{Maya}} game (notes by {{Hikoe Enomoto}})(in
  {{Japanese}}),
\newblock {\em RIMS K\^oky\^uroku} {\bf 98} (1970), 105--135.

\bibitem{sato-maya-1970}
M.~Sato,
\newblock On {{Maya}} game (notes by {{Hikoe Enomoto}})(in {{Japanese}}),
\newblock {\em Sugaku no Ayumi} {\bf 15:1} (1970), 73--84.

\bibitem{sprague-uber-1935}
R.~P. Sprague,
\newblock {{\"U}}ber mathematische {{Kampfspiele}},
\newblock {\em Tohoku Math. J.} {\bf 41} (1935), 438--444.

\bibitem{stembridge-eigenvalues-1989}
J.~Stembridge,
\newblock On the eigenvalues of representations of reflection groups and wreath
  products,
\newblock {\em Pacific J. Math.} {\bf 140:2} (1989), 353--396.

\bibitem{welter-theory-1954}
C.~P. Welter,
\newblock The theory of a class of games on a sequence of squares, in terms of
  the advancing operation in a special group,
\newblock {\em Indag. Math.} {\bf 57} (1954), 194--200.

\end{thebibliography}
\end{document}